\documentclass[11pt]{amsart}
\usepackage[dvipsnames]{xcolor}
\usepackage[left=1in,top=1in,right=1in,bottom=1in]{geometry}
\usepackage{mathtools, amssymb, tikz, mathrsfs}
\usepackage[T1]{fontenc}
\usepackage{tikz, mathtools}
\usepackage[colorlinks=true, linkcolor=red!80!black, urlcolor=purple, citecolor=blue!70!black]{hyperref}
\usetikzlibrary{matrix,arrows}

\usepackage{tikz-cd}

\newtheorem{question}{Question}
\newtheorem{theorem}{Theorem}[section]
\newtheorem{lemma}[theorem]{Lemma}
\newtheorem{corollary}[theorem]{Corollary}
\newtheorem{proposition}[theorem]{Proposition}
\newtheorem{prop}[theorem]{Proposition}

\theoremstyle{definition}
\newtheorem{definition}[theorem]{Definition}
\newtheorem{example}[theorem]{Example}
\newtheorem{remark}[theorem]{Remark}
\newtheorem{conjecture}[theorem]{Conjecture}

\theoremstyle{remark}

\newtheorem{construction}[theorem]{Construction}

\def\owlset@enDeligne-Mumfordark{\ensuremath{\scriptscriptstyle\blacksquare}}
\newcommand{\F}{{\mathbb F}}

\newcommand{\PP}{{\mathbb P}}
\newcommand{\Q}{{\mathbb Q}}

\newcommand{\Z}{{\mathbb Z}}

\def\bbar#1{\setbox0=\hbox{$#1$}\dimen0=.2\ht0 \kern\dimen0 
\overline{\kern-\dimen0 #1}}

\newcommand{\fp}{{\mathfrak p}}
\newcommand{\fq}{{\mathfrak q}}

\newcommand{\fA}{{\mathfrak R}}
\newcommand{\fB}{{\mathfrak B}}
\newcommand{\fD}{{\mathfrak D}}
\newcommand{\fM}{{\mathfrak M}_{\Gamma}}
\newcommand{\fMb}{{\mathfrak M}_{\Gamma,b}}

\newcommand{\fU}{{\mathfrak U}}
\newcommand{\fX}{{\mathfrak X}}

\newcommand{\calA}{{\mathcal A}}

\newcommand{\calC}{{\mathcal C}}
\newcommand{\calD}{{\mathcal D}}
\newcommand{\calE}{{\mathcal E}}
\newcommand{\calF}{{\mathcal F}}
\newcommand{\calG}{{\mathcal G}}
\newcommand{\calH}{{\mathcal H}}
\newcommand{\calI}{{\mathcal I}}

\newcommand{\calL}{{\mathcal L}}
\newcommand{\calM}{\mathcal M}

\newcommand{\calO}{{\mathcal O}}
\newcommand{\calP}{{\mathcal P}}

\newcommand{\calU}{{\mathcal U}}

\newcommand{\calW}{{\mathcal W}}
\newcommand{\calX}{{\mathcal X}}
\newcommand{\calY}{{\mathcal Y}}

\newcommand{\logct}{\Omega_{X^\nu}^{[1]}(\log(D^\nu + \Delta^{\nu}_{dl}))}

\newcommand{\dblv}{\Delta^\nu_{dl}}
\newcommand{\dbl}{\Delta_{dl}}

\DeclareMathOperator{\supp}{supp}

\DeclareMathOperator{\im}{im}

\DeclareMathOperator{\Aut}{Aut}

\DeclareMathOperator{\Sym}{Sym}

\DeclareMathOperator{\Spec}{Spec}

\newcommand{\OD}{\Omega^{[1]}_X(\log D)}

\newcommand{\isom}{\simeq}

\newcommand{\setm}{\smallsetminus}

\newcommand{\sF}{\mathscr{F}}

\begin{document}
%%%
%%%
%%%
\title{Hyperbolicity and uniformity of varieties of log general type}
\author{Kenneth Ascher}
\email{kascher@princeton.edu}
\author{Kristin DeVleming}
\email{kdevleming@ucsd.edu}
\author{Amos Turchet}
\email{amos.turchet@sns.it}

\begin{abstract} Projective varieties with ample cotangent bundle satisfy many notions of hyperbolicity, and one goal of this paper is to discuss generalizations to quasi-projective varieties. A major hurdle is that the naive generalization is false -- the log cotangent bundle is never ample. Instead, we define a notion called almost ample which roughly asks that it is as positive as possible. We show that all subvarieties of a quasi-projective variety with almost ample log cotangent bundle are of log general type. In addition, if one assumes globally generated then we obtain that such varieties contain finitely many integral points. In another direction, we show that the Lang-Vojta conjecture implies the number of stably integral points on curves of log general type, and surfaces of log general type with almost ample log cotangent sheaf are uniformly bounded.  \end{abstract}

\maketitle

\vspace*{-1em}
\section{Introduction}\label{sec:intro}
One major area of research in algebraic and arithmetic geometry is a conjectural circle of ideas due to e.g. Green, Griffiths, and Lang, asserting that various notions of hyperbolicity coming from algebraic, differential, and arithmetic geometry coincide:  Brody hyperbolicity (no entire holomorphic curves), arithmetically hyperbolicity (finiteness of rational points), and an algebraic notion (all subvarieties are of general type). In general, it is  quite challenging to find examples of varieties satisfying these hyperbolicity notions. Somewhat classical work shows that varieties with positive cotangent bundle form a large class of examples for which hyperbolicity is known. In particular, if $X$ is a smooth projective variety and the cotangent bundle $\Omega^1_X$ is ample, then $X$ is Brody hyperbolic (and thus Kobayashi hyperbolic), and all subvarieties of $X$ are of general type. If $\Omega^1_X$ is additionally globally generated, then $X$ has finitely many rational points.

The first goal of this paper is to generalize these results to quasi-projective varieties. In this case, the aforementioned conjectures predict the equivalence of Brody hyperbolicity, all subvarieties being of log general type, and finiteness of integral points. If we view a quasi-projective variety $V$ as the complement of an snc divisor in a smooth projective variety $V \cong X \setminus D$, then naively one would hope that assuming the log cotangent bundle $\Omega^1_X(\log D)$ is ample would be the correct condition to guarantee hyperbolicity. It turns out that this is too much to hope for, as the following shows. 

\begin{proposition}[see Proposition \ref{rmk:neverample}] If $\dim X \geq 2$ and $D$ is not empty, then  $\Omega^1_X(\log D)$ is never ample. \end{proposition}

Therefore we see that the extension to the logarithmic setting is in fact much more intricate. In fact, positivity always fails along the divisor $D$. Motivated by this, we instead ask that the log cotangent sheaf is \emph{almost ample}. Roughly speaking, we say  $\Omega^1_X(\log D)$ is \emph{almost ample} if it is as positive as possible (see Definition \ref{def:dbarample} for a precise definition; see also \cite{brotbek}). Essentially we require that $\Omega^1_X(\log D)$ is \emph{big}, and its non-ample locus is contained in the support of $D$.

 \subsection{Subvarieties and almost ample log cotangent bundles} We first show that all subvarieties of quasi-projective varieties with almost ample log cotangent bundle are of log general type.

\begin{theorem}[see Corollary \ref{cor:ampleness}] Let $(X,D)$ be a log smooth pair with almost ample $\Omega^1_X(\log D)$. Then all pairs $(Y, E)$, where $E := (Y \cap D)_{red}$, with $Y \subset X$ and $Y$ not contained in $D$ are of log general type. \end{theorem}

This result is a consequence of the following result, which shows the result holds for singular pairs whose singularities are those coming from the minimal model program.

\begin{theorem}[see Theorem \ref{thm:ampleness}; see Corollary \ref{cor:slcample} for the slc extension]\label{thm:introample}  
Let $(X,D)$ be log canonical pair with almost ample log cotangent $\Omega^{[1]}_X(\log D)$. Then all pairs $\big(Y, E)$, where $E := (Y \cap D)_{red}$, with $Y \subset X$ such that $Y$ is neither contained in $D$ nor $\textrm{Sing(X)}$ are of log general type. \end{theorem}

We note that we can also prove a slightly stronger result for the case of surfaces, since the singular locus of an lc surface consists only of points.

\begin{corollary}[see Remark \ref{prop:ampleness}]
Let $(X,D)$ be log canonical surface pair with almost ample log cotangent $\Omega^{[1]}_X(\log D)$. Then all pairs $\big(Y, E)$, where $E := (Y \cap D)_{red}$, with $Y \subset X$ irreducible and not contained in $D$ are of log general type.
\end{corollary}

The main tools used to solve these problems involve the theory of extending reflexive differentials on singular spaces (e.g. \cite{diff, GK}) and the work of Campana-P\u{a}un \cite{cp}. Although one often restricts to (log) smooth varieties when studying hyperbolicity questions, we are forced to consider the situation when $(X,D)$ has lc (or more generally slc) singularities to obtain the uniformity results alluded to in the abstract (see Section \ref{sec:introunif}). Before proceeding, we make a few remarks concerning almost ample log cotangent sheaves. 

\begin{remark}\leavevmode
\begin{enumerate}
\item When $X$ is smooth, our notion does not quite coincide with almost ample as in \cite[Definition 2.1]{brotbek}. If the log cotangent sheaf is almost ample in the sense of \cite{brotbek}, then it is almost ample in our sense. However, the assumptions in our definition are a priori weaker.
\item Work of \cite{brotbek} shows that for any smooth projective $X$, there exists a divisor $D$ so that $\Omega^1_X(\log D)$ is almost ample. 
\item For a log smooth pair $(X,D)$ with almost ample $\Omega^1_X(\log D)$, the complement $X \setminus D$ is Brody hyperbolic.
\end{enumerate}
\end{remark}

\subsection{Positivity and arithmetic hyperbolicity}

It is also natural to ask whether positivity on the log cotangent sheaf implies finiteness statements for integral points. Indeed, Moriwaki proved that projective varieties $X$ over a number field $K$ with ample and globally generated $\Omega^1_X$ have finitely many $K$ points (see Theorem \ref{thm:amplecot2}). We prove the following generalization.

\begin{theorem}[see Theorem \ref{thm:moriwakimain}]\label{thm:moriwaki1} Let $V$ be a smooth quasi-projective variety with log smooth compactification $(X, D)$ over a number field $K$. If the log cotangent sheaf $\Omega^1_X(\log D)$ is globally generated and almost ample, then for any finite set of places $S$ the set of $S$-integral points $V(\calO_{K,S})$ is finite.
\end{theorem}

Alternatively, Theorem \ref{thm:moriwaki1} can also be seen as a consequence of the log cotangent being almost ample (using Theorem \ref{thm:introample}) and the following.

\begin{theorem}[see Theorem \ref{thm:moriwakisub}]\label{thm:moriwaki2} Let $V \cong (X \setminus D)$ be a log smooth variety over $K$. If the log cotangent sheaf $\Omega^1_X(\log D)$ is globally generated, then for any finite set of places $S$, every irreducible component of $\overline{V(\calO_S)}$ is geometrically irreducible and isomorphic to a semi-abelian variety. \end{theorem}

The main tools used to prove the above results are the theory of semi-abelian varieties and the quasi-Albanese, as well as Vojta's generalization \cite{vojta1} of Faltings' theorem on subvarieties of abelian varieties to semi-abelian varieties.

Finally, we wish to discuss the following example, which discusses hyperbolicity in families. We show that the property of having all subvarieties of log general type is \emph{not} a closed condition. Moreover, we stress that this property may not hold in the special fiber even if every curve outside of the double locus is of log general type. 

\begin{remark}(See Example \ref{ex:counterexample}) We show the existence of a stable family $(X,D) \to B$ where $B$ is a curve, the generic fiber $(X_\eta, D_\eta) \subset \PP^3$ is a smooth hyperbolic surface, while the special fiber $(X_0, D_0)$  is not. It contains a single curve which is \emph{not} of log general type, and this curve is contained in the double locus.  \end{remark}

\subsection{Uniformity}\label{sec:introunif}
Recall that (one form of) Lang's conjecture (see Conjecture \ref{conj:BL}) predicts that varieties of general type do not contain a Zariski dense set of rational points. One of the most intriguing consequences of Lang's Conjecture, proved in \cite{CHM}, is that the number of rational points on curves of genus $g \geq 2$ over a number field $K$ is not only finite, but is also bounded by a constant depending only on $g$ and $K$ (see Theorem \ref{conj:unif}). 
 The original ideas of \cite{CHM} have since been extended and generalized leading to proofs that similar uniformity statements, conditional on the Lang Conjecture, hold in higher dimensions (see Section \ref{subsec:rat} for more details).

A natural question addressed in \cite{Aell} is the following: does the Lang-Vojta Conjecture (Conjecture \ref{conj:LV}) imply similar uniformity statements for integral points? Abramovich showed this cannot hold unless one restricts the possible models used to define integral points. This led Abramovich to define the notion of \emph{stably integral points} (see Definition \ref{def:stab_c}) to prove uniformity results, conditional on the Lang-Vojta Conjecture, for integral points on elliptic curves, and together with Matsuki in \cite{AM} for integral points on principally polarized abelian varieties (PPAVs). 

\begin{remark} Roughly speaking, stably integral points are integral points which remain integral after stable reduction (see Definition \ref{def:stab_c} for a precise definition). \end{remark}

 Thus the final goal of this paper is to generalize uniformity statements for stably integral points to pairs of log general type (see Theorem \ref{thm:mainthm}). We generalize \cite{Aell} (elliptic curves with punctured origin) and  related results can be found in \cite{CHM} ($g \geq 2$), and e.g. \cite{evertse} in the case of $g = 0$ with marked points. Unless stated otherwise, $K$ will denote a number field, and $S$ a finite set of places of $K$ containing the Archimedean ones. 

 \begin{theorem}[see Theorem \ref{th:curves}]
  \label{th:u_curve}
Assume the Lang-Vojta Conjecture \ref{conj:LV}. If $(C,D)$ is an irreducible stable pointed curve over $K$, then the set of stably $S$-integral points on $C$ is uniformly bounded.
\end{theorem}

To obtain analogous results in higher dimensions one needs to tackle an extra problem: stable models do not exist due to the absence of a semistable reduction theorem over arbitrary Dedekind domains.  As a result, there is no canonical choice of model for higher dimensional algebraic varieties. 

Instead, using recent results on moduli of stable pairs (see Definition \ref{def:stab_pair}), the higher dimensional analogue of the moduli of stable pointed curves, we define ``good'' models which play the role of stable models for curves, and actually generalize them (see Section \ref{sec:models}). This allows us to define a notion of \emph{moduli-stably-integral points} ($ms$-integral points, see Definition \ref{def:ms}), integral points which are integral with respect to a fixed model of the moduli space.

The second result we obtain is that $ms$-integral points on families of stable pairs lie in a subscheme whose degree is uniformly bounded, which generalizes a theorem of Hassett \cite[Theorem 6.2]{Hassett}. 

\begin{theorem}[see Corollary \ref{cor:subscheme}]  \label{th:deg}
  Assume the Lang-Vojta Conjecture \ref{conj:LV}. Suppose that $f:(X,D)\to B$ is a stable family over a smooth variety $B$ with integral, openly canonical (see Definition \ref{def:olc}), and log canonical general fiber over $K$. For all $b \in B(K)$, there exists a proper closed subscheme $A_b$ containing all $ms$-integral points of $X_b$ whose irreducible components have uniformly bounded degree. 
\end{theorem}

From here, one wishes to conclude uniformity for higher dimensional varieties. The first obstacle to proving uniformity, which also appears in the case of rational points, is the presence of curves with non-positive Euler characteristic, which contain infinitely many integral points. One natural approach to circumvent this problem, as motivated above, is to ask for positivity of the log cotangent sheaf. From here, the standard way to conclude uniformity from Theorem \ref{th:deg}, is to run an induction argument once you answer ``yes'' to the following question, thus overcoming the second obstacle:

\begin{question} Do $ms$-integral points satisfy the \emph{subvariety property} -- i.e. are $ms$-integral points for a pair $(X,D)$ lying on a pair $\big(Y,E:= Y \cap D\big)$ with $Y \subset X$ also $ms$-integral points for $(Y,E)$?  \end{question}

E.g., if $(X,D)$ is a quasi-projective surface, it is not a priori clear that an $ms$-integral point of $(X,D)$ which lies on a curve $(Y, E)$, where $E = Y \cap D$ is also an $ms$-integral point for $(Y,E)$!  \\

The above question was answered affirmatively for abelian surfaces using N\'eron models \cite{AM}. Without having an explicit model to work with in higher dimensions, we reduce the question to a sufficient geometric criterion. Since the notion of stably integral points (or more generally $ms$-integral points) requires a compact moduli space, even if we are interested in uniformity for \emph{smooth} varieties, we are forced to consider degenerations of our smooth variety to \emph{singular} ones. If $\dim X > 2$, then a degeneration of $X$ can have a non-normal locus of positive dimension, and it is a highly non-trivial problem to show that there are no non-hyperbolic components contained in this locus. 

However, assuming positivity of the log cotangent sheaf, we are able to verify the subvariety property for surfaces (see Section \ref{sec:descent}), and thus prove uniformity.

\begin{theorem}[see Corollary \ref{cor:unif_2}]\label{thm:mainthm}
  \label{th:d2}
Assume the Lang-Vojta Conjecture. Let $(X,D)$ be a log canonical stable surface pair with good model $(\calX, \calD)$ such that \begin{enumerate}
\item $D = \sum D_i$ is an effective $\Q$-Cartier divisor with $K_X + D$ ample and
\item each fiber of $(\calX, \calD)$ has almost ample log cotangent sheaf (see Definition \ref{def:dbarample}),
\end{enumerate}
then there exists a constant $N = N(K,S,v)$ where $v$ is the volume of $(X,D)$, such that the set of $ms$-integral points of $(X,D)$ has cardinality at most $N$, i.e.
  \[
    \# (X \setm D)(\calO_{K,S}^{ms}) \leq N = N(K,S,v).
  \]
\end{theorem}

One main ingredient in this paper, following the ideas of \cite{CHM}, is a Fibered Power Theorem, proved in \cite{fpt} (see Theorem \ref{thm:fpt}), which gives the analogue for pairs of the main Theorem of \cite{Afpt}.

\begin{remark}\leavevmode
\begin{enumerate}
\item While we make a choice of model, we show that (up to changing the constants involved), the results do not depend on choice of model (see Remark \ref{remark:models}).
\item Assuming the Lang-Vojta conjecture, our methods give a proof for uniformity under \emph{any assumption} that guarantees that all subvarieties are of log general type. Asking for \emph{almost ample log cotangent} is natural from a geometric standpoint (see Remark \ref{rmk:uniformity}). \end{enumerate}\end{remark}

\subsection*{Appendices} In Appendix \ref{sec:stacks}, we define the stack of stable pairs over $\Q$. This is probably known, but we include it for lack of reference. In Appendix \ref{app:sheaves} we show there exists an almost ample log cotangent sheaf on the universal family of the moduli stack. Appendix \ref{app} gives an alternative definition of $ms$-integral points that does not depend on the choice of models of stacks. In Appendix \ref{sec:apphyp} we present Example \ref{ex:counterexample}.

\begin{remark} This paper originally appeared with the title ``Uniformity for integral points on surfaces, positivity of log cotangent sheaves and hyperbolicity.'' \end{remark}

\subsection*{Acknowledgements}
We thank Brendan Hassett for suggesting this problem and Dan Abramovich for his constant support. This paper has benefited from discussions with Jarod Alper, Dori Bejleri, Damian Brotbek, Ya Deng, Gabriele Di Cerbo, Carlo Gasbarri, J\'anos Koll\'ar, S\'andor Kov\'acs,  Wenfei Liu, Yuchen Liu, Steven Lu, Zsolt Patakfalvi, Fabien Pazuki, S\"onke Rollenske, David Rydh, Karl Schwede, Jason Starr, and Bianca Viray. We thank Ariyan Javanpeykar for discussions leading to the notion of \emph{good models}, Max Lieblich for help verifying Example \ref{ex:counterexample}, and Erwan Rousseau for pointing us to \cite{cp}. We thank the referees for their comments and suggestions which greatly improved this paper. Research of AT supported in part by funds from NSF grant DMS-1553459. Research of KA supported in part by funds from NSF grant DMS-1162367/11500528 and an NSF Postdoctoral fellowship. 

\section{Previous results}\label{sec:previous}
In this section we discuss previous results on hyperbolicity and uniformity.

\subsection{Some hyperbolicity results}
A motivating conjecture usually associated to Green, Griffiths, and Lang is the following:

\begin{conjecture} Let $X$ be a projective geometrically integral variety over a number field $K$. The following are equivalent:
\begin{enumerate}
\item $X$ is arithmetically hyperbolic,
\item $X_{\mathbb{C}}$ is Brody hyperbolic, and 
\item every integral subvariety of $X$ is of general type.
\end{enumerate}
\end{conjecture}

As mentioned in the introduction, it is somewhat classical that varieties with positive cotangent bundle enjoy many hyperbolicity properties (see e.g. \cite{debarre}). In particular, we have the following.

\begin{theorem}\label{thm:amplecot1} If $X$ is a smooth projective variety with ample $\Omega^1_X$ then \begin{enumerate}
\item all subvarieties of $X$ are of general type (see \cite[6.3.28]{laz2}), and 
\item $X$ is Brody (and thus Kobayashi) hyperbolic (see \cite[3.1]{demailly}).
\end{enumerate} \end{theorem}

It is expected that the rational points are also finite, and in this direction we have the following.

\begin{theorem}\cite[Theorem E]{moriwaki}\label{thm:amplecot2}
If $X$ is a smooth variety defined over a number field $k$ with ample and globally generated $\Omega^1_X$ then the set $X(k)$ is finite.
\end{theorem}

We do note that Theorem \ref{thm:amplecot1} is \emph{not} an if and only if. That is, there exist hyperbolic varieties whose cotangent bundle is not ample. The standard example is a product of high genus curves -- this variety is Brody hyperbolic and all subvarieties are of general type, but the cotangent is not ample.

\subsection{Uniformity for rational points}
\label{subsec:rat}
%%%
Faltings proved that for projective curves $\calC$  over $K$ of genus $g=g(\calC)\geq 2$, the set $\calC(K)$ is finite \cite{Falt}. In higher dimensions there is a conjectural analogue:

\begin{conjecture}(Bombieri-Lang (surfaces), Lang ($\dim X > 2$), \cite{Lang} and \cite{Lang2})
  \label{conj:BL}
Let $X$ be an algebraic variety defined over $K$. If $X$ is of general type, then the set $X(K)$ is not Zariski-dense.
\end{conjecture}

\cite{CHM} showed that Conjecture \ref{conj:BL} implies that $\#\calC(K)$ in Faltings' Theorem is not only finite, but is also uniformly bounded by a constant $N=N(g,K)$ that does \emph{not} depend on the curve $\calC$.

\begin{theorem}[see \cite{CHM}]
\label{conj:unif}
  Let $K$ be a number field and $g \geq 2$ an integer. Assume Lang's Conjecture. Then there exists a number $B=B(K,g)$ such that for any smooth curve $\calC$ defined over $K$ of genus $g$ the following holds: $ \# \calC(K) \leq B(g,K) $
\end{theorem}

 Pacelli \cite{Pacelli} (see also \cite{Aquadratic}), proved that $N$ only depends on $g$ and $\deg(K:\Q)$. More recently, cases of Theorem \ref{conj:unif} have been proven unconditionally (\cite{krz}, \cite{stoll} and \cite{paz}) depending on the Mordell-Weil rank of the Jacobian of the curve and for \cite{paz}, on an assumption related to the Height Conjecture of Lang-Silverman. It has also been shown that families of curves of high genus with a uniformly bounded number of rational points in each fiber exist \cite{dnp}. 

Na\"ive translations fail in higher dimensions as subvarieties can contain infinitely many rational points. However, one can expect that after removing such subvarieties the number of rational points is bounded. Hassett proved that for surfaces of general type this follows from Conjecture \ref{conj:BL}, and that the set of rational points on surfaces of general type lie in a subscheme of uniformly bounded degree \cite{Hassett}.

The idea behind both proofs is the following: consider a family $f: X \to B$ whose general fibers are general type curves (resp. surfaces) over a base $B$, the proof reduces to showing that the number of rational points in the fibers is uniformly bounded. If the total space of the family is itself a variety of general type, the Lang Conjecture trivially implies the uniformity statement. In general this is not the case, but if the family has maximal variation in moduli, then the dominant irreducible component of a high enough fibered power $X_B^n \to B$ will be of general type. Conjecture \ref{conj:BL} and an induction argument will then give uniformity. In general, it is always true that for $n$ big enough $X_B^n$ admits a dominant map to a variety of general type. From this, one can conclude the result in a similar fashion. This can be applied to a ``global'' family of curves to obtain the result of \cite{CHM}.

The algebro-geometric result alluded to above is known as a \emph{fibered power theorem} and was shown for curves in \cite{CHM}, for surfaces \cite{Hassett} and in general by Abramovich \cite{Afpt}. The pairs analogue is Theorem \ref{thm:fpt} (\cite{fpt}).
In higher dimensions, similar uniformity statements hold conditionally on Lang's Conjecture, and follow from the fibered power theorem under some additional hypotheses that take care of the presence of subvarieties that are not of general type (\cite{AV}).\subsection{Uniformity of Integral Points}\label{subsec:int}
The analogue of Faltings' Theorem for quasi-projective curves is Siegel's Theorem-- every affine curve of positive Euler characteristic possesses a finite number of $S$-integral points. There is a conjectural generalization to higher dimensions, that extends Lang's Conjecture \ref{conj:BL} to the quasi-projective case:

\begin{conjecture}(Lang-Vojta)\label{conj:LV}
 Let $X$ be a quasi-projective variety and let $\calX \to \Spec \calO_{K,S}$ be a model over the $S$-integers. If $X$ is of openly log general type (see Definition \ref{def:loggt}), then $\calX(\calO_{K,S})$ is not Zariski dense.
\end{conjecture}

A natural question, is whether Conjecture \ref{conj:LV} implies a uniform bound on the set of $S$-integral points for quasi-projective curves of openly  log general type (see Definition \ref{def:loggt}). Abramovich (\cite[0.3]{Aell}) gave a counterexample: he constructed an elliptic curve, where the number of $S$-integral points in the complement of the origin grow arbitrarily when one changes the model. However, imposing minimality conditions on the model leads to statements similar to \cite{CHM}. In particular, if one considers \emph{stable} models for the quasi-projective curves $E \setm O_E$ over a number field, then the cardinality of the set of $S$-integral points of this model, called \emph{stably}-integral points, is uniformly bounded, conditional on Conjecture \ref{conj:LV}. This was extended to PPAVs of $\dim 2$ \cite{AM}.

Both results rely on the existence of good models for elliptic curves and abelian varieties. While this can be extended to arbitrary stable curves, it is not clear how to define stable models in arbitrary dimensions outside of the abelian case. However, it was observed by Abramovich and Matzuki that stably integral points admit a nice moduli interpretation (see Section \ref{sec:models}).

Unconditional results for uniformity of integral points in certain classes of curves, coming from Thue Equations, were proved in \cite{LT}, given some bound on the Mordell-Weil rank of the Jacobian. 

As Vojta's Conjecture implies Lang's Conjecture, and thus a uniform version of the conjecture, one can ask if Vojta's conjecture implies a uniformity statement for heights. This was shown by Ih for curves \cite{ih}, and was generalized to certain families of hyperbolic varieties \cite{ari}.

%%%
%%%
\section{Preliminaries and Notations}\label{sec:prel}
%%%
%%%
The ring of $S$-integers, i.e. the set $\{ x \in K: \| x \|_v \leq 1, \forall v \notin S \}$ will be denoted by $\calO_{K,S}$.

\begin{definition} 
 Given an algebraic variety $X$ defined over $K$, a \emph{model} of $X$ over $\calO_{K,S}$ is a separated scheme $\calX$ together with a flat map $\calX \to \Spec \calO_S$ of finite type such that the generic fiber is isomorphic to $X$, i.e. $X \cong \calX \times_{\Spec \calO_S} \Spec K$. 
\end{definition}

Given a quasi-projective variety we will use the following definition for openly log general type.

\begin{definition}(see \cite[Definition 1.3]{fpt})
  \label{def:loggt}
  A quasi-projective variety $X$ is \emph{of openly log general type} if there exists a desingularization $\widetilde{X} \to X$ and a projective compactification $\widetilde{X} \subset Y$ with $D = Y \setm \widetilde{X}$ a divisor of normal crossings, such that $\omega_Y(D)$ is big.
\end{definition}

The above definition is independent of both the choice of desingularization and of the compactification.
From both the viewpoint of birational geometry and of integral points on quasi projective varieties, it has become natural to consider pairs of a variety and a divisor. 

\begin{definition} 
 A \emph{pair} $(X,D)$ is the datum of a projective variety $X$ and a $\mathbb{Q}$-divisor  $D = \sum d_i D_i$ which is a linear combination of distinct prime (Weil) divisors.
\end{definition}

\begin{remark}\label{rmk:loggt}We will often say that a pair $(X,D)$, with $X$ a projective variety and $D$ a normal-crossings divisor is \emph{of openly log general type} if the quasi-projective variety $X \setm D$ is. In some applications we will require some conditions on the singularities of the pair.\end{remark}

\begin{definition} A pair $(X,D)$ has \emph{log canonical singularities} (or is lc) if $X$ is normal, $K_X +D$ is $\Q$-Cartier, and there is a log resolution (see \cite[Notation 0.4(10]{kom}) $f: Y \to X$ such that
\[
  K_Y  + \sum a_E E = f^*(K_X + D) \]
where all the $a_E \leq 1$ and the sum goes over all irreducible divisors on $Y$. The pair has \emph{canonical singularities} if all $a_E \leq 0$.   
\end{definition}

\begin{definition}
  \label{def:olc}
  A log canonical pair $(X,D)$ is \emph{openly canonical} if $X\setm D$ has canonical singularities.
\end{definition}

For an lc pair $(X,D)$ that is openly canonical, being of openly log general type is equivalent to the condition that  $\omega_X(D)$ is big, in particular this can be checked without referring to a log-resolution of singularities as in Definition \ref{def:loggt}. 

\begin{definition}\label{def:reflexive}
  Let $\sF$ be a coherent sheaf on a scheme $X$. The dual of $\sF$ is denoted by $\sF^*$, and we define the sheaf $\sF^{**}$ to be the \emph{reflexive hull} of $\sF$. We say that $\sF$ is \emph{reflexive} if the natural map $\sF \to \sF^{**}$ is an isomorphism. If $X$ is S2, then the \emph{reflexive powers} of $\sF$ are the reflexive hulls of the tensor powers, i.e. $\sF^{[m]} := (\sF^{\otimes m})^{**}$. We then define the \emph{reflexive symmetric power} of $\sF$ to be $\Sym^{[a]}(\sF) := (\Sym^a(\sF))^{**}$.
\end{definition}

We note that the notion of a model extends naturally to pairs:

\begin{definition} Consider a pair $(X,D)$ with $D$ a $\Q$-Cartier divisor over $K$. A \emph{model} for $(X,D)$ over $\Spec \calO_S$, is a model $\calX$ of $X$ together with an (effective) $\Q$-Cartier divisor $\calD \to \calX$ whose restriction to the generic fiber is isomorphic to $D$. In other words, a model for $(X,D)$ is the datum of a model $\calX$ of $X$ and a compatible model $\calD$ of $D$. 
\end{definition}

Models of pairs can be used to define integral points with respect to a divisor.

\begin{definition}
\label{def:integral}
Consider a pair $(X,D)$, with $D$ a Cartier divisor, and a model $(\calX,\calD)$ over $\Spec \calO_{K}$. An $(S,D)$-\emph{integral point} is a section $P: \Spec \calO_{K} \to \calX$ such that the support of $P^*\calD$ is contained in $S$. An $S$-integral point of a quasi-projective variety $X\setm D$ is an $(S,D)$-integral point for the pair $(X,D)$.
\end{definition}

\subsection{Stability}\label{subsec:stab}
As mentioned in Section \ref{subsec:int}, in order to obtain uniformity results for integral points, it is necessary to restrict the possible models under consideration. We recall here some  definitions that will be useful later. First we need a crucial definition.

\begin{definition}
  \label{def:slc}
  A pair $(X, D = \sum d_i D_i)$ is \emph{semi-log canonical (slc)}  if $X$ is reduced and $S_2$, the divisor $K_X + D$ is $\Q$-Cartier and the following hold:
\begin{enumerate}
\item $X$ is Gorenstein in codimension one, and
\item if $\nu : X^{\nu} \to X$ is the normalization, then the pair $(X^{\nu}, D^\nu + \Delta_{dl}^{\nu})$ is log canonical, where $D^\nu$ denotes the preimage of $D$ and $\Delta_{dl}^{\nu}$ denotes the preimage of the double locus $\Delta_{dl}$ on $X^{\nu}$.
\end{enumerate}
\end{definition}

Slc pairs with ample log canonical sheaf generalize stable pointed curves to higher dimension.

\begin{definition}\label{def:stab_pair} A pair is \emph{stable} if the $\Q$-Cartier line bundle $\omega_X(D)$ is ample, and the pair is semi-log canonical. A \emph{stable family} is a flat family $(X,D) \to  B$ over a normal variety $B$ such that 
\begin{enumerate}
\item $D$ avoids the generic and codimension one singular points of every fiber,
\item $\omega_f(D)$ is $\Q$-Cartier, and 
\item $(X_b, D_b)$ is a stable pair for all $b \in B$.
\end{enumerate}
\end{definition}

We end this subsection by discussing fiber powers of families of stable pairs.
\begin{definition}
  Given a stable family $\pi: (X,D) \to B$, denote by $(X_B^n,D_n)$ the \emph{$n^{th}$ fibered power} of $X$ over $B$, where $X_B^n$ is the (unique) irreducible component of the fiber power which dominates $B$, and
$D_n := \pi_1^* D + \dots + \pi_n^* D,$
where $\pi_i: X_B^n \to X$ is the $i$-th projection.
\end{definition}

\begin{theorem}[see \cite{fpt}]\label{thm:fpt} Let $(X,D) \to B$ be a stable family such that general fiber is integral, openly canonical (see Definition \ref{def:olc}), and log canonical over a smooth projective variety $B$. Then there exists an integer $n > 0$, a positive dimensional pair $(W, \Delta)$ of openly  log general type (see Definition \ref{def:loggt} and Remark \ref{rmk:loggt}), and a morphism $(X^n_B, D_n) \to (W, \Delta)$. \end{theorem}

\begin{remark}\label{rmk:moduli} The \emph{moduli space of stable pairs} $\calM_{\Gamma}$ is constructed and proven to have projective coarse space in \cite{kp}. It requires a choice of invariants $\Gamma = (n, v, I)$, where $n$ is the dimension of the pairs, $v$ is their volume, and $I$ is a coefficient set satisfying the DCC condition. \end{remark}

%%%

\section{Positivity of the log cotangent sheaf -- subvarieties}\label{sec:log_ct}
As we saw in Theorem \ref{thm:amplecot1}, if $X$ is smooth and projective with $\Omega^1_X$ ample, then all subvarieties of $X$ are of general type. The goal of this section is to prove a similar result for quasi-projective varieties, namely finding a positivity condition that guarantees that all subvarieties are of log general type. As noted in the introduction, the immediate generalization fails -- the log cotangent sheaf is never ample (see Proposition \ref{rmk:neverample}). Nevertheless, one can still obtain some hyperbolicity properties assuming the log cotangent sheaf is, in a sense, as positive as possible (see Theorem \ref{thm:ampleness}, Corollary \ref{cor:slcample}, Corollary \ref{cor:ampleness}, and Remark \ref{prop:ampleness}).

In the following subsection we investigate positivity properties of the log cotangent sheaf, its consequences for the geometry of subvarieties, and its extensions to non-normal varieties. We remind the reader that we are forced to consider the singular situation to obtain the uniformity results mentioned in Section \ref{sec:introunif}. 

\subsection{Definition of the log cotangent sheaf} We begin by collecting some facts about the log cotangent sheaf for log canonical pairs. Like the cotangent sheaf, it is not clear how to define the log cotangent sheaf for non-normal varieties which appear naturally when considering families of normal varieties. To this end, we use formalism and ideas from \cite{diff} and \cite{GK}. We begin by recalling some facts about \emph{reflexive differentials} following \cite{diff} (see also Definition \ref{def:reflexive}). 

\begin{definition}(see \cite[Notation 2.16]{diff}) Suppose $(X,D)$ is a log canonical pair. Consider the open set $U \subset X$ whose complement is the singular locus of $X$. Denote by $i: U \to X$ the inclusion. Then the \emph{sheaf of reflexive differentials} is $\Omega_X^{[1]}(\log D) := i_*(\Omega^1_U (\log D))$. \end{definition}

This sheaf is reflexive and torsion free but does not need to be locally free. However, the following theorem asserts that any logarithmic 1-form defined on the smooth locus $U \subset X$ can be extended to a logarithmic 1-form possibly with poles on certain exceptional divisors on any resolution of singularities. 

\begin{theorem}[see Theorem 1.4 of \cite{GK}]\label{thm:gk}
Let $(X,D)$ be a log canonical pair, and let $\pi: \widetilde{X} \to X$ denote a log resolution of $(X,D)$. Then the sheaf $\pi_* \Omega^1_{\widetilde{X}}(\log\widetilde{D})$ is reflexive, where $\widetilde{D}$ is any reduced divisor such that \[ \mathrm{Exc}(\pi) \wedge \pi^{-1}(D) \subseteq \supp \left \lfloor{\widetilde{D}}\right \rfloor  \subseteq \pi^{-1}(\left \lfloor{D}\right \rfloor ),\]  $\mathrm{Exc}(\pi) \wedge \pi^{-1}(D)$ denotes the largest divisor contained in both $\mathrm{Exc}(\pi)$ and $\pi^{-1}(D)$.\end{theorem}

\subsection{Positivity properties of the log cotangent sheaf}

As mentioned above, the sheaf $\Omega^{[1]}_X(\log D)$ is in general only a coherent reflexive sheaf, therefore we recall here the notion of ampleness in this context borrowing ideas from both \cite{vie} and \cite{kp} (see also \cite{kollarsubadd}).

\begin{definition}(See \cite[Definition 3.7]{kp})
  \label{def:ample}
  Let $\calF$ be a coherent sheaf on a normal and reduced quasi-projective variety $X$ and let $\calH$ be an ample line bundle on $X$.
  \begin{enumerate}
  \item We say that $\calF$ is \emph{ample} if there exists a positive integer $a > 0$ such that the sheaf\\ $\Sym^{[a]} \calF \otimes \calH^{-1}$
  is globally generated. 
  \item We say that $\calF$ is \emph{big} if there exists a positive integer $a > 0$ such that the sheaf\\ $ \Sym^{[a]} \calF \otimes \calH^{-1}$
   is generically globally generated.
  \end{enumerate}
\end{definition}

\begin{remark}These definitions are independent of the choice of $\calH$ (see \cite[Lemma 2.14.a]{vie}). \end{remark}

As mentioned earlier, the log cotangent sheaf is \emph{never ample} (see \cite[Section 2.3]{brotbek}).

\begin{prop}\label{rmk:neverample} Let $(X,D)$ be a log smooth pair with $\dim X \geq 2$ and $D$ non-empty. The sheaf $\Omega^1_X(\log D)$ is never ample. \end{prop}
\begin{proof} Suppose that $\Omega^1_X(\log D)$ were ample. Consider the following exact sequence (see \cite[Proposition 2.3]{ev}): 
$$ 0 \to \Omega^{1}_X \to \Omega^{1}_X(\log D) \to \oplus_{i \in I} \calO_{D_i} \to 0.$$
Consider the restriction of this sequence to a component $D_i \subseteq D$, and tensor the above sequence with $\calO_{D_i}$ to obtain a surjection:
$$A\to \calO_{D_i} \oplus Q\to 0,$$
where $A$ is an ample sheaf (the restriction of an ample sheaf), and $Q$ is a torsion sheaf supported at $D_i \cap D_j \neq \emptyset$ for all $j \in I$ such that $i \neq j$.
However, since $\calO_{D_i} \oplus Q$ is not ample, and $\dim D_i \geq 1$, there cannot exist such a surjection from an ample sheaf, and so $\Omega^1_X(\log D)$ can never be ample.
\end{proof}

In light of the above, we are led to a definition that captures the strongest positivity assumption one can make on $\OD$, even for non-normal $X$.  Given an slc pair, after normalizing, we arrive at an lc pair $(X^\nu, D^\nu + \dblv)$. A natural condition is thus to assume positivity of $\logct$.

\begin{remark}\label{rmk:uniformity} 
The following Definition \ref{def:dbarample} is a natural assumption in light of both Proposition \ref{rmk:neverample}, as well as the aforementioned work of Greb-Kebekus-Kov\'acs-Peternell (esp. \cite[Example 3.2]{diff}), to exclude the existence of subvarieties which are not of log general type in a stable pair $(X,D)$. It has the advantage of being geometric in nature and natural when considering stable pairs and their moduli.
We note that there are examples of moduli spaces of stable pairs for which each object does not contain subvarieties of log general type, e.g. the moduli space of pointed stable curves, and the moduli space of semiabelic pairs (\cite{Alexeev2}). In particular, \emph{any positivity condition which guarantees the non-existence of subvarieties of log general type will allow us to prove uniformity} (Theorem \ref{thm:mainthm}). 
\end{remark}

Before giving our definition of almost ample, we will recall the definition of augmented base loci.

\begin{definition}\cite[Definition 2.4]{baselocus}
Let $X$ be a normal projective variety, let $\sF$ be a coherent sheaf, and $A$ an ample line bundle. Let $r = p/q \in \Q >0$. The \emph{augmented base locus} of $\sF$ is \[ \mathbf{B}_+(\sF) := \displaystyle\bigcap_{r \in \Q >0} \mathbf{B}(\sF - rA),\] where $\mathbf{B}(\sF-rA) = \mathbf{B}(\Sym^{[q]}\sF \otimes A^{-p})$.\end{definition}

\begin{remark}\leavevmode
\begin{enumerate}
    \item The augmented base locus does not depend on the choice of an ample divisor $A$. 
    \item By \cite[Proposition 3.2]{baselocus}, if $\sF$ is a coherent sheaf and  $\pi: 
\PP(\sF) = \PP(\Sym \sF) \to X$ is the canonical morphism, then $\pi(\mathbf{B}_+(\calO_{\PP(\sF)}(1))) = \mathbf{B}_+(\sF)$, i.e. the non-ample locus of $\sF$.
\end{enumerate}
\end{remark}

\begin{definition}\label{def:dbarample} Let $(X,D)$ be a log canonical  pair. We say that the log cotangent  sheaf of $(X,D)$  is \emph{almost ample} if:
\begin{enumerate}
\item $\Omega^{[1]}_{{X}}(\log D)$ is big, and
\item $\mathbf{B}_+\big(\Omega^{[1]}_{{X}}(\log D)\big) \subseteq \textrm{Supp}(D)$, 
\end{enumerate} 
where $\mathbf{B}_+(\calF)$ denotes the augmented base locus of $\calF$. If $(X,D)$ is an slc pair, then we say the log cotangent sheaf of $(X,D)$ is almost ample if the log cotangent sheaf of the normalization $(X^\nu, D^\nu + \dblv)$ is almost ample.

\end{definition}

\begin{remark}\label{rmk:almostample}\leavevmode
\begin{enumerate}
\item When $X$ is smooth, our notion does not quite coincide with almost ample as in \cite[Definition 2.1]{brotbek}. If the log cotangent sheaf is almost ample in the sense of \cite{brotbek}, then it is almost ample in our sense. However, the assumptions in our definition are a priori weaker. 

 \item \cite{brotbek} shows that for smooth projective $X$, there always exists a $D$ so that $\Omega^1_X(\log D)$ is almost ample. In particular, one expects many moduli spaces of stable pairs in which (at least) the smooth objects have almost ample log cotangent.
 \item For a log smooth pair $(X,D)$ with almost ample $\Omega^1_X(\log D)$, the complement $X \setminus D$ is Brody hyperbolic by Definition \ref{def:dbarample} (2) (see e.g. \cite[Proposition 3.3]{demailly}).
   \item When $D = \emptyset$, ``almost ample'' (log) cotangent implies that all subvarieties of $X$ are of general type.
 \end{enumerate} 
 \end{remark}
 
 Before preceding, we recall a result of Campana-P\u{a}un.  
 
 \begin{theorem}\cite[Theorem 4.1]{cp}\label{thm:cp} Let $X$ be a smooth projective variety, and $\Delta \subseteq X$ a reduced divisor with at worst normal
crossing singularities. If some tensor power of $\Omega^1_X(\log \Delta)$ contains a subsheaf with big determinant,
then $K_X + \Delta$ is big. \end{theorem}
 
 We will now show that positivity of the log cotangent on a log smooth pair $(X,D)$ implies that all subvarieties (not contained in $D$) are of log general type.
 
\begin{theorem}\label{thm:ampleness} 
Let $(X,D)$ be log canonical pair with almost ample log cotangent $\Omega^{[1]}_X(\log D)$. Then all pairs $\big(Y, E)$, where $E := (Y \cap D)_{red}$, with $Y \subset X$ such that $Y$ is neither contained in $D$ nor $\textrm{Sing(X)}$ are of log general type.

\end{theorem}

\begin{proof} Consider a log resolution $(\widetilde{X}, \widetilde{D}) \to (X, D)$ and let $(\widetilde{Y}, \widetilde{E})$ be the strict transform of $Y$ and the pullback of $E$ with reduced structure.  Let $(\overline{Y}, \overline{E})$ be a log resolution of $(\tilde{Y}, \tilde{E})$.  The composition gives a map $\phi: (\overline{Y}, \overline{E}) \to (X,D)$. Since $Y$ is not contained in $D$, by the definition of almost ample, $Y$ is not contained in the base locus of $\Omega^{[1]}_X(\log D)$.  Using Theorem \ref{thm:gk}, this gives a map $\phi^*(\Omega^{[1]}_X(\log D)) \to \Omega_{\overline{Y}}(\log \overline{E})$. The image of this map is a big subsheaf of $\Omega^1_{\overline{Y}}(\log \overline{E})$, being a quotient of a big sheaf, and thus its determinant is also big. By Theorem \ref{thm:cp}, $K_{\overline{Y}} + \overline{E}$ is big, and so $(Y,E)$ is of log general type. \end{proof}

\begin{remark}
In the above proof, we used that given Definition \ref{def:ample}, the quotient of a big sheaf is a big sheaf. We stress that this is not true for other notions of big. The key idea is to be big there is a generically surjective map, and this map remains generically surjective when restricting to a subvariety not contained in the base locus (in this case a subvariety contained in the divisor $D$).   
\end{remark}

This result extends to semi-log canonical pairs by applying Theorem \ref{thm:ampleness} to the normalization.  

\begin{corollary}\label{cor:slcample}
If $(X,D)$ is an slc pair with almost ample log cotangent, then any irreducible subvariety of $X$ not contained in either $D$ or $\textrm{Sing}(X)$ is of log general type.
\end{corollary}

Finally, we can state a stronger result if $(X,D)$ is log smooth.

\begin{corollary}\label{cor:ampleness}Let $(X,D)$ be a log smooth pair with almost ample $\Omega^1_X(\log D)$. Then all pairs $(Y,E)$, where $E := (Y \cap D)_{red}$, with $Y \subset X$ and $Y$ not contained in $D$ are of log general type. \end{corollary}

As we saw in the introduction, we can use Corollary \ref{cor:ampleness} to prove Theorem \ref{thm:moriwaki1} on finiteness of integral points. However, for Theorem \ref{thm:mainthm} (uniformity), we will need the result for slc pairs, and also need control over subvarieties contained in the singular locus. We are only able to control such subvarieties for surfaces (see Remark \ref{prop:ampleness} and Corollary \ref{cor:subvarietyprop}). 

\begin{remark}
\label{prop:ampleness}
In the case of surfaces, we can prove Theorem \ref{thm:ampleness} directly without using \cite{cp}.  Further, since the only singular points of an lc surface are points, we can prove a slightly stronger statement. Let $(X,D)$ be log canonical surface pair with almost ample log cotangent $\Omega^{[1]}_X(\log D)$. Then all pairs $\big(Y, E)$, where $E := (Y \cap D)_{red}$, with $Y \subset X$ irreducible and not contained in $D$ are of log general type.
 We include the proof below. 
\end{remark}

\begin{proof}[Proof of Remark \ref{prop:ampleness}]
Since $\Omega^{[1]}_X (\log D)$ is big with prescribed base locus, its restriction to a curve $Y \subset X$ irreducible and not contained in $D$ is also big. Since $Y$ is a curve, big is equivalent to ample. Consider the normalization $\phi: Y^\nu \to Y$, and denote by $E^\nu = \phi^{-1}(E)$.  As $\Omega^{[1]}_X (\log D) \mid_Y$ is ample, its pullback $\phi^*(\Omega^{[1]}_X(\log D) \mid_Y)$ is big. By Theorem \ref{thm:gk} (see \cite[Theorem 1.4]{GK}), there is a generically surjective map
 \[ \phi^*(\Omega^{[1]}_X (\log D) \mid_Y) \to \Omega^1_{Y^\nu}(\log E^\nu) = \calO_{Y^\nu}(K_{Y^\nu} + E^\nu)\] 
 so that $K_{Y^\nu} + E^\nu$ is big, and thus $(Y,E)$ is of log general type. \end{proof}

%%%

\section{Generalizations of Moriwaki's results and finiteness of integral points}\label{sec:moriwaki}
In \cite{moriwaki}, Moriwaki proved that for smooth projective varieties over number fields $K$ with globally generated cotangent bundle, every irreducible component of $\overline{X(K)}$ is geometrically irreducible and isomorphic to an abelian variety. Moreover, if in addition the cotangent bundle is ample then there are only finitely many $K$ points.  We now show that these results generalize for integral points on log smooth surfaces by replacing ``ample cotangent'' with ``almost ample log cotangent''. We stress that the ideas in the proofs below are the ideas of Moriwaki, and we are simply applying them to our newer framework. Let $V$ denote a smooth quasi-projective surface with log smooth completion $(X,D)$,  let $\calA_V$ denote the \emph{quasi-Albanese} variety, and let $\alpha: V \to \calA_V$ denote the quasi-Albanese morphism (see \cite[Section 2.7]{fujinoalb}). 

\begin{theorem}\label{thm:moriwakimain}Let $V$ be a smooth quasi-projective variety with log smooth compactification $(X, D)$ over a number field $K$. If the log cotangent sheaf $\Omega^1_X(\log D)$ is globally generated and almost ample, then for any finite set of places $S$ the set of $S$-integral points $V(\calO_{K,S})$ is finite.\end{theorem}

\begin{proof}[Proof]
Assume that $Y$ is an irreducible component of $\overline{V(\calO_S)}$ with $\dim Y \geq 1$ and completion $(\overline{Y}, \overline{E})$. Since $\Omega^1_X(\log D)$ is almost ample and globally generated, its restriction to $\overline{Y}$ is as well. As in \cite[Lemma 2.3]{moriwaki}, we first show $\dim(\calA_Y) \geq 2 \dim Y$.  Indeed, the proof of \cite[Lemma 2.3]{moriwaki} holds verbatim by defining the sheaf $L = \im \left( \mu^*(\Omega^{1}_{X}(\log D)) \to \Omega^1_{Y'}(\log E') \right)$, where $\mu: Y' \to \overline{Y}$ is an appropriate resolution.  The generically surjective map $\mu^*(\Omega^{1}_{X}(\log D)) \to \Omega^1_{Y'}(\log E')$ exists by Theorem \ref{thm:gk}, and the sheaf $L$ is big, globally generated, and locally free of rank $\dim Y$ by choice of resolution $\mu$, and the proof of \cite[Lemma 2.3]{moriwaki} applies directly. 

Finally, by \cite[Theorem 1.1]{moriwaki}, because $\dim(\calA_Y) \geq 2 \dim Y$, we have that $Y(\calO_S)$ is not dense in $Y$, which is a contradiction.  Indeed, Moriwaki's proof of \cite[Theorem 1.1]{moriwaki} works by replacing $\alpha$ by the quasi-Albanese morphism (see \cite{fujinoalb}) and by replacing Faltings' theorem (\cite[Theorem A]{moriwaki}) by Vojta's theorems (\cite{vojta1, vojta3}). \end{proof}

\begin{theorem}\label{thm:moriwakisub} Let $V \cong (X \setminus D)$ be a log smooth variety over $K$. If the log cotangent sheaf $\Omega^1_X(\log D)$ is globally generated, then for any finite set of places $S$, every irreducible component of $\overline{V(\calO_S)}$ is geometrically irreducible and isomorphic to a semiabelian variety. \end{theorem}

\begin{proof}[Proof]
  Since $\Omega^1_X(\log D)$ is assumed to be globally generated,  its restriction to $V$, namely $\Omega^1_V$ is as well, so that there is a surjection $H^0(V, \Omega^1_V) \otimes \calO_V \to \Omega^1_V.$ By \cite[Lemma 3.12]{fujinoalb}, $H^0(V, \Omega^1_V) \otimes \calO_{\calA_V} \cong \Omega^1_{\calA_V}$, so pulling back by $\alpha: V \to \calA_V$ gives a surjection $\alpha^*(\Omega^1_{\calA_V}) \to \Omega^1_V$. Therefore, similarly to \cite[Theorem B]{moriwaki}, every irreducible component of $\overline{V(\calO_S)}$ is isomorphic to a semi-abelian variety. Indeed, Vojta's Theorem (\cite{vojta1, vojta3}) implies that every irreducible component of a semi-abelian variety containing infinitely many integral points is a translate of a semi-abelian subvariety. Therefore the same argument as in Moriwaki's proof applies since smooth \'etale covers of semi-abelian varieties are semi-abelian varieties by \cite[Theorem 4.2]{fujinoalb}.
\end{proof}

We now shift gears to consider uniformity, ultimately using these hyperbolicity results to obtain uniformity for surfaces. We begin with the case of curves.

%%%%%%

\section{Uniformity for log stable curves}\label{sec:curves}
Abramovich observed \cite{Aell} that uniformity statements for integral points on curves cannot hold without any restrictions on the model (see \cite[0.3]{AM} - for an example and discussion), and instead one should consider stable models. Such a choice provides good models (see section \ref{sec:models} for a more general framework that extends to higher dimension) which possess ``positivity'' properties preventing the appearance of non-hyperbolic components in the model. As a result, such positivity does not allow the number of integral points to grow arbitrarily. Abramovich's notion of stably integral points for the complement of the origin in an elliptic curve can be easily generalized to any stable pointed curve $(C,D)$ as follows.
 
\begin{definition}
\label{def:stab_c}
Let $(C,D)$ be a stable pointed curve defined over $K$. A point $P \in (C\setminus D)(K)$ is called a \emph{stably} $(S,D)$-integral point if there exists a finite extension $L \supset K$ and a stable model $(\calC, \calD)$ over $\Spec \calO_{L,S_L}$ such that $P$ is a $(S_L,\calD)$-integral for $\calC$. Equivalently $P \in (\calC \setminus \calD) (\calO_{L,S_L})$.
\end{definition}

Given Definition \ref{def:stab_c} one can ask whether the same uniformity results proved in \cite{Aell} hold more generally for a pair $(C,D)$ of openly log general type. To prove this we introduce the following:

\begin{definition}
  \label{def:correl}
 Let $(C,D) \to B$ be a family of pointed stable curves over $K$. Given a subset $\calP \subset C(K)$, denote by $\calP^n \subset C^n_B$ the $n^{th}$ fibered power of $\calP$ over $B$. Then $\calP$ is \emph{$n$-correlated} if there exists an $n > 0$ such that $\calP^n$ is contained in a proper closed subset of $C^n_B$.
\end{definition}

The importance of $n$-correlated sets for uniformity questions is apparent in the following:

\begin{lemma}
  Let $C \to B$ a family of projective irreducible curves and let $\calP$ be an $n$-correlated subset of $C(K)$. There exists a nonempty open set $U \subset B$ and an $N \in \mathbb{N}_{>0}$ such that for every $b \in U$, $\calP \cap C_b \leq N$.
  \label{lemma:n-corr}
\end{lemma}
\begin{proof}
  See \cite[Lemma 1.1]{CHM} , \cite[Lemma 1]{Aell}, or \cite[Lemma 1.1.2]{AM}.
\end{proof}

In view of Lemma \ref{lemma:n-corr}, in order to prove uniformity for stably integral points on curves of openly log general type, we need to prove that the set of stably integral points is $n$-correlated. We start by stating the following lemma, which will be important throughout.

\begin{lemma}  \label{lemma:spread}
  Let $(X,D)$ be a pair defined over $K$ and let $\phi: (X,D) \to (W,E)$ be a dominant morphism. Then, given a proper model $(\calX,\calD)$ over $\calO_{K,S}$, there exists $S' \supset S$ and a model $(\calW,\calE)$ over $\calO_{K,S'}$ such that $\phi$ extends to a map $(\calX,\calD)_{\Spec \calO_{K,S'}} \to (\calW,\calE)$.
\end{lemma}
\begin{proof}
  The extension property follows from spreading out techniques, noting that the extensions to models of $X$, $W$, $D$ and $E$ are compatible after possibly enlarging $S$.
\end{proof}

Recall that a stable pair (Definition \ref{def:stab_pair}) in dimension one is a projective curve with at worst nodal singularities and a reduced divisor disjoint from the singular locus. Therefore we obtain:

\begin{proposition}
  \label{prop:stably_curves_correl}
  Let $\pi: (C,D) \to B$ be a family of stable pointed curves with smooth general fiber and let $\calP$ be the set of stably integral points of the family. Then the Lang-Vojta Conjecture implies that $\calP$ is $n$-correlated for some $n$ large enough.
\end{proposition}
\begin{proof}
  Since the family is a stable family with smooth general fiber, we may apply Theorem \ref{thm:fpt} to obtain a positive dimensional pair $(W,E)$ of openly  log general type and a dominant morphism $(C^n,D_n) \to (W,E)$  which restricts to a regular map on the complement $C^n \setm D_n \to W \setm E$. Applying Conjecture \ref{conj:LV} to $(W,E)$ implies that there exists a proper closed subset $V \subset W$ containing all the $(S,E)$-integral points. By Lemma \ref{lemma:spread} there exists a proper closed subset $F_n \subset C^n$ containing all the $(S',D_n)$-integral points. Thus $\calP^n \subset F_n$.
\end{proof}

Using Noetherian induction on the base we can prove the following:

\begin{theorem}
Assume the Lang-Vojta Conjecture. For all irreducible stable pointed curves $(C,D)$ defined over $K$, the number of stably $S$-integral points on $C$ is uniformly bounded.
  \label{th:curves}
\end{theorem}

\begin{proof}
  Following \cite{CHM}, we apply Proposition \ref{prop:stably_curves_correl} to a ``global'' family of stable pointed curves as follows: let $g$ be the genus of $C$; by assumption we may assume that $C$ is irreducible. The stability condition implies that $2g -2 + \deg D$ is a positive integer. In particular, there exists $l$, which does not depend on $C$, such that each stable curve $C'$ of genus $g'$ and reduced divisor $D'$ with $g' = g$ and $\deg D' =\deg D$, can be embedded in $\PP^N$ using the linear system $\lvert l K_{C'}(D') \rvert$. The theory of Hilbert Schemes gives the existence of a family $\pi: X \to B$ with $\deg D$ sections defined over $K$ such that given any pair $(C',D')$ as before, there exists a $K$-rational point $b \in B(K)$ such that $X_b \cong (C',D')$. This family can be constructed by taking the closure of the locus of pluri-log canonical curves and its restriction to the universal family in the corresponding Hilbert scheme.
  By construction, the general fiber of $X \to B$ is a smooth curve of openly log general type (in particular it is stable). Therefore, Proposition \ref{prop:stably_curves_correl} applies, so the set of stably $S$-integral points of $\pi: X \to B$ is $n$-correlated for $n$ large enough. By Lemma \ref{lemma:n-corr}, this implies the existence of an open subset $U \subset B$ such that for every $K$-rational $b \in U$, there exists a non-negative integer $N = N(K,S,g,\deg D)$ such that
  \[
    X_b(\calO_{K,S}^{\text{stably}}) \leq N = N(K,S,g,\deg D).
  \]
  Finally one applies Noetherian induction on the dimension of $B \setm U$ to obtain a similar bound for \emph{all} fibers in the family $\pi: X \to B$. More explicitly one defines $B_1$ to be the union of all irreducible components of $B \setm U$ whose generic point is a smooth curve and considers the corresponding restricted family $X_1 \to B_1$. Applying Lemma \ref{lemma:n-corr} to this new family gives the existence of on open set $U_2 \subset B_1$ where the stably integral points of the fibers are uniformly bounded, possibly by a different constant $N_1$. This inductively gives a chain of base schemes $B_i$ such that $\dim B_i < \dim B_{i-1}$ and therefore stabilizes after a finite number of steps. For each $B_i$ one has a uniform bound given by a constant $N_i$ outside an open subset $U_{i+1}$. Taking $N$ to be the maximum of all $N_i$ shows that the stably integral points of \emph{any} fiber are at most $N$. Since the family has been chosen to be global, it follows that such a bound holds for all stable curves, and thus this proves the theorem.
\end{proof}

\section{Good models} \label{sec:models}
In this section we extend the notion of stably integral points to higher dimensions. First we need to construct models that play the role of stable models, since the latter are not known to exist for $\dim \geq 2$. The main idea behind our construction is to fix models for the moduli stack of stable pairs and construct the models of the pairs as base changes of the models of the stacks, noting that models for stacks are completely analogous to those of varieties. The starting point is the following observation of Abramovich-Matsuki that gives a nice moduli interpretation of stably integral points. We let $\overline{\calA}_{g,1}$ denote the universal family over $\overline{\calA}_g$. 
%\begin{proposition}[see \cite[Proposition 3.1.3]{AM}]\label{prop:AM}
\begin{proposition}[see Proposition  3.1.3 of \cite{AM}]\label{prop:AM}
  Let $(A,\Theta)$ be a PPAV defined over $K$ and let $P \in (A \setm \Theta)(K)$. Consider the associated moduli map
  $
    P_m: \Spec K \to (A,\Theta) \to (\overline{\calA}_{g,1},\mathbf{\Theta}).
  $
 Then $P$ is stably $S$-integral if and only if $P_m$ is an $S$-integral point in $\overline{\calA}_{g,1}\setm \mathbf{\Theta}$.
\end{proposition}

This implies that one way to characterize stably $S$-integral points is to look at their image in an appropriate moduli space and test their integrality with respect to a model of such a moduli stack.

Contrary to the moduli space of PPAVs, the moduli space of stable pairs has not been defined over $\Spec \mathcal{O}_{K,S}$. We remedy this by fixing a model over a Dedekind domain and show that our results are independent of the choice of such a model.

\subsection{Construction of good models}
By Appendix \ref{sec:stacks} the moduli stack of stable pairs can be defined as a Deligne-Mumford stack over $\Q$, with projective coarse moduli space. We now make a choice of models of such stacks as follows: choose a ring of integers $\fA$ of a number field $K_\fA$ and models of  $\calM_\Gamma,\calU$ and $\calD$ over $\fA$ such that in the following diagram
\[ 
  \begin{tikzcd} (\calU,\calD) \ar[r]^{} \ar[d]_{} & (\fU,\fD) \ar[d]^{} \\ 
  \calM_\Gamma \ar[r]^{} \ar[d] & \fM \ar[d] \\
   \Spec K_\fA \ar[r]^{} & \Spec \fA \end{tikzcd}
\] 
$\fU,\fD$ and $\fM$ are proper stacks over $\fA$ that are models for $\calU,\calD$ and $\calM_{\Gamma}$ over $\fA$. Note that the existence of the diagram follows from Section \ref{sec:stacks} and the definition of models. We can define models for stable pairs with respect to the choices of the models and of $\fA$ as follows:

\begin{definition}
\label{def:good}
 Given a stable pair $(X,D)$ defined over $K$, let $(\calX, \calD)$ be any model over $\Spec\calO_{K,S}$. We say that $(\calX, \calD)$ is a \emph{good model} (with respect to the choices of the moduli stacks, the ring of integers $\fA$ and the models of the stacks) if there exists a number field $L \supset K$ such that ${\calO_{L,S_L}} \supset \fA$, where $S_L$ is the set of places lying over $S$, and
      \[
	(\calX_{\calO_{L,S_L}}, \hspace{1ex} \calD_{\calO_{L,S_L}}) \isom (\Spec {\calO_{L,S_L}} \times_{\fM} \fU, \hspace{1ex} \Spec {\calO_{L,S_L}} \times_{\fM} \fD),
      \]
      where $\calX_{\calO_{L,S_L}}$ and $\calD_{\calO_{L,S_L}}$ are the base change through $\Spec {\calO_{L,S_L}} \to \Spec \calO_{K,S}$. We say $(\calX,\calD)$ is defined over $\calO_{L,S_L}$ if there is a number field $L$ such that the above isomorphism holds.
\end{definition}

Good models play the role of stable models in dimension one and are the key ingredient in the definition of moduli stably integral points for stable pairs.

\begin{definition}
\label{def:ms}
 Let $(X,D)$ be a stable pair over $K$. A rational point $P$ is \emph{moduli stably $S$-integral} ($ms$-integral for short when the reference to $S$ is clear), if there exists a finite extension $L \supset K$ and a good model $(\calX, \calD)$ over $\Spec \calO_{L,S_L}$ such that the image of the map:
  \[
    \begin{tikzcd} P: \Spec L \ar[r] &(X_L, D_L) \ar[r] \ar[d] & (\calX_{\calO_{L,S_L}}, \hspace{1ex} \calD_{\calO_{L,S_L}}) \ar[d] \\%& (\Spec \calO_{L,S_L}\times_\fM \fU,  \hspace{1ex} \Spec \calO_{L,S_L}\times_\fM \fD) \ar[d]\\
    & \Spec L \ar[r] & \Spec \calO_{L,S_L} %\ar@{=}[r] & \Spec \calO_{L,S_L} 
 \end{tikzcd}
  \]
  is $(S_L,D_L)$-integral in the good model $(\calX_{\calO_{L,S_L}}, \hspace{1ex} \calD_{\calO_{L,S_L}})$. We will denote the set of all moduli stably $S$-integral points of $(X,D)$ as $X(\calO_{K,S}^{\text{ms}})$.
\end{definition}

\begin{remark} \label{remark:models}For any two choices of models of the stacks over two different rings $\fA$ and $\fA'$, we can always find a ring of integers $\fB$ containing $\fA$ and $\fA'$  such that the base change of any of the two models will define a model over $\fB$ and any $ms$-integral point with respect to $\fA$ and $\fA'$ is integral with respect to $\fB$. In particular the results of the following sections of this paper do not depend on the choice of the model, up to extending the ring of integers that we consider, and possibly adjusting the constants involved (see e.g. \cite[Appendix B]{Rydh}).
\end{remark}
The notion of $ms$-integral point extends the notion of stably integral points both for curves and PPAVs. In fact, the following more general observation holds.

\begin{proposition}
  \label{prop:ms-stacks}
  Let $\fX$ be a proper Deligne-Mumford stack representing a functor of stable pairs and admitting a universal family $\fU$ and universal divisor $\fD$ such that $\fX$ and $(\fU,\fD)$ have models over a ring of integers $\fA$. Suppose that  $\calO_{K,S} \supset \fA$, and let $(X,D)$ be a pair defined over $K$ which admits the following diagram
  \[
    \begin{tikzcd} (X,D) \ar[d] \ar[r] & (\calX,\calD) \ar[r] \ar[d] & (\fU,\fD) \ar[d] \\ \Spec K \ar[r]  & \Spec \calO_K \ar[r] & \fX \end{tikzcd}
  \]
  where $(\calX,\calD) \isom (\Spec \calO_K \times_\fX \fU, \hspace{1ex} \Spec \calO_K \times_\fX \fD)$ is a good model. If $P \in (X\setm D)(K)$ is an $ms$-integral point of $X$, then the image of $P$ in $(\fU,\fD)$ is $(\fD,\calO_{K,S})$-integral.
\end{proposition}

\begin{proof} 
  $P$ is $ms$-integral if it is integral in $(\calX,\calD)$, which is the base change of the universal family to the ring of integers $\calO_{K,S}$. This implies that the image of $P$ is integral in $(\fU,\fD)$.
\end{proof}

Thus $ms$-integral points agree with the corresponding notion in \cite{Aell} and \cite{AM}. 
\begin{remark}
We see that moduli stably integral points in dimension one for the moduli functor of stable curves, $\fA = \Z$ and $\calM_{\Gamma}$ the stack $\overline{\calM}_{g,n}$, are stably integral points for curves. In fact, given a stable curve $\calC$ over $K$, if $L \supset K$ is the finite extension where $\calC$ acquires stable reduction then 
  \[
    \Spec \calO_{L,S_L} \times_{\calM_{g,n}} \fU
  \]
  is a stable model for $\calC$ and a good model according to Definition \ref{def:good}. Therefore moduli stably integral points for these choices coincide with stably integral points for curves.
  The same holds for the moduli functor of principally polarized quasi-abelian schemes (or abelic pairs \'{a} la Alexeev) \cite{Alexeev2} and $\fA = \Z$. In this case a ``good'' model is a stable quasi-abelian model and $ms$-integral points are stably integral points of \cite{AM}. 
\end{remark}

We stress that whenever the existence of a moduli stack of stable pairs over any finitely generated $\Z$ algebra with modular interpretation is known, the definitions above (and the consequential results to follow) will, in particular, be valid for this choice of model. Moreover, up to enlarging $S$, the model of the universal family $(\fU,\fD) \to \fM$ will be a stable family for an appropriate definition of singularities for fibers with residue field of characteristic $p>0$.

In Appendix \ref{app} we present an alternative approach that, although leading to weaker results, does \emph{not} depend on the existence of such a model used in this section.

\section{Uniform bound for the degree of the subscheme containing $ms$-integral points}\label{sec:deg_bound}
Given the definition of $ms$-integral points (Definition \ref{def:ms}), one could ask what kind of uniformity statements can follow from the Lang-Vojta Conjecture. In this section, we prove that for a stable family with nice singularities, $ms$-stably integral points lie on a subscheme whose degree is uniformly bounded. The following is inspired by Hassett \cite{Hassett}.

\begin{theorem}
  \label{th:deg_bound}
Assume the Lang-Vojta Conjecture. Let $(X,D)\to B$ be a stable family defined over $K$ with integral, openly canonical, and log canonical generic fiber over a smooth projective variety $B$. There exists an open set $U \subset B$ such that for all rational points $b$ in $U(K)$, there exists a proper subscheme $A_b$ containing all $ms$-integral points of $X_b$ such that, if $N(b)$ is the sum of the degrees (in a suitable projective embedding) of the components of $A_b$, then $N(b)$ is uniformly bounded.
\end{theorem}

\begin{proof}
By Theorem \ref{thm:fpt}  there exists a positive integer $n$ and a positive dimensional pair $(W, \Delta)$ of openly  log general type and a dominant morphism $(X_B^n, D_n) \to (W,\Delta)$, which induces a regular map $X_B^n \setm D_n \to W\setm \Delta$. Assuming the Lang-Vojta Conjecture for any model of $(W,\Delta)$, there exists a proper closed subvariety of $W \setm \Delta$ containing all the $(\Delta,S)$-integral points.
    
 Define $Z_n$ to be the preimage of this subvariety, which by definition and by Lemma \ref{lemma:spread} contains all $S$-integral points of $X_B^n$ and therefore the $ms$-integral points of the fibered power (note that by Proposition 4.5 of \cite{fpt}, the pair $(X_B^n,D_n)$ is a stable pair and by construction, the definition of $ms$-integral points are compatible with taking fiber powers). Define, by induction, closed subvarieties of $X_B^j$ for each $1\leq j \leq n$ in the following way:

 \begin{enumerate}
      \item For each $j$, denote by $\pi_j : X_B^j \to X_B^{j-1}$ and by $\pi_{ij} : X_B^j \to X_B^i$ the projection morphisms;
      \item For each $k=1,\dots,j$,  denote by $\pi_{j,1}^k : X_B^j \to X$ the $k$-th projection;
      \item For each $j$, denote by $Z_j$ the maximal closed subset of $X_B^j \setm D^j$ such that 
	\[
	  \pi_{nj}^{-1}(Z_j)\setm \sum_{k>j} \pi_{n,1}^{k*}D \subset Z_n;
	\]
      \item For each $j$, denote by $U_j$ the complement of $Z_j$ in $X_B^j \setm D_j$. By construction, $U_n$ does not contain any $ms$-integral points. Moreover, by maximality of $Z_j$ for each $j$, the preimage $\pi_j^{-1}(u)$ is not contained in $Z_j$ for every $u \in U_{j-1}$.
 \end{enumerate}
 Note that by definition, for $k \leq j -1$ we have that  
\[
\pi_{j-1,1}^k \circ \pi_j = \pi_{j,1}^k,
\]
and for each $j$ we have that
\[
  D_j = \pi_j^* D_{j-1} + (\pi_{j,1}^{j})^* D.
\]

 Then one has that 
\[
  \pi_j^{-1}(Z_{j-1}) \setm (\pi_{j,1}^{j})^* D \subset Z_j.
\]
 For every $u \in U_{j-1}$, its inverse image in $X_B^j$ intersects $Z_j$ in a proper subvariety (since $\pi^{-1}_j(u)$ is not contained in $Z_j$). Call $A_j$ this subvariety (which might not be of pure dimension) and let $d_j = \sum \deg(A_j)$, in a suitable projective embedding. Let $N = \max_j d_j$ and $b$ be a $K$-rational point of $B$, then we claim $X_b(\calO_{K,S}^{\text{ms}})$ lies in a subvariety of degree $\leq N$.
 To prove the claim we define an index $\hat{\jmath}$ as
\[
  \hat{\jmath} = \min \{ j: U_j \cap X_b^j(\calO_{K,S}^{\text{ms}}) = \emptyset\}.
\]
Note that by construction $\hat{\jmath} \leq n-1$. Now pick a rational point $u$ in $U_{\hat{\jmath}}$ which lies in $X_b^{\hat{\jmath}}(\calO_{K,S}^{\text{ms}})$; this in particular implies that $\pi_{\hat{\jmath}}^{-1} (u) = X_b$. By the above discussion $\pi_{\hat{\jmath}}^{-1} (u) (\calO_{K,S}^{\text{m}}) \subset A_{\hat{\jmath}}=: A_b$, which is a subvariety of degree $d_{\hat{\jmath}} \leq N$. This proves the claim, and thus the theorem.
\end{proof}

\begin{corollary}\label{cor:subscheme}
  Assume the Lang-Vojta Conjecture. In the notation above, for every fiber $X_b$ the set of $ms$-integral points $X_b(\calO_{K,S}^{ms})$ is contained in a proper subscheme $A_b$. Moreover, the sum of the degrees of the irreducible components of $A_b$ are uniformly bounded. 
\end{corollary}
\begin{proof}
 For the family $(X,D) \to B$, Theorem \ref{th:deg_bound} gives a uniform bound $N_0$ for the degree of the subscheme $A_b$ in all fibers $X_b$ with $b$ in an open subset of $B$. Consider the restricted family $(X_1,D_1) \to B_1$ where $B_1$ is the union of the irreducible components of $B$ with integral, openly canonical, and log canonical generic fiber. Applying Theorem \ref{th:deg_bound} gives a bound $N_1$ for $A_b$ in all fibers with $b$ in an open subset of $B_1$. Iterating this, we find integers $N_0,\dots,N_m$ (finitely many by being Noetherian), such that for every fiber $X_b$, the sum of the degrees of the irreducible components of the proper subscheme $A_b$ is bounded by the max of the $N_i$.
\end{proof}

%%%%%%%%
%%%%%%%%

Given the degree bound obtained in Theorem \ref{th:deg_bound} one would like to conclude, assuming the Lang-Vojta Conjecture, that $ms$-integral points in a stable pair satisfy some uniformity. However, one cannot expect a result as strong as Theorem \ref{th:curves}, since a stable pair can contain curves with negative Euler characteristic (and thus contain infinitely many $ms$-integral points).  Even if we are in a case where the $ms$-integral points are finite, there is still another problem to be tackled, namely the \emph{subvariety property} mentioned in the introduction.

\begin{definition}\label{def:subvarietycondition} Let $(X,D)$ be a stable pair over $K$, and let $P$ be an $ms$-integral point. We say that $P$ satisfies the \emph{subvariety property} if for all irreducible subvarieties $Y$ containing $P$ but not contained in $D$, we have that $P$ is an $ms$-integral point for $(Y,D_Y)$ with $D_Y = Y \cap D$. We say that $ms$-integral points satisfy the \emph{subvariety property} if every $ms$-integral point of $(X,D)$ does. \end{definition}

To prove that $P$ is $ms$-integral for $(Y,D_Y)$, we have to exhibit a good model of the pair, possibly after extending  $K$. If $(Y,D_Y)$ is stable, we can construct a good model (up to extending $K$) using the moduli map (Definition \ref{def:good}), but there is a priori no well defined map between the models of $X$ and $Y$.  Since we do not have explicit models for $X$ and $Y$, we instead consider a stronger geometric condition that would imply the subvariety property. 

As we saw in Section \ref{sec:log_ct} (e.g. Theorem \ref{thm:ampleness}), a natural way to guarantee some sort of hyperbolicity is to assume that the log cotangent sheaf is almost ample. It turns out that there is still work to do -- assuming the log cotangent sheaf is almost ample is only enough to guarantee the subvariety property is true \emph{outside the double locus} of an slc pair. 

Nevertheless, in the following sections, we will show for surfaces that $ms$-integral points behave well for subvarieties of pairs with almost ample log cotangent sheaf. To obtain uniformity results we need to show that these positivity conditions can be defined at the level of the moduli stack. More precisely, we prove the existence of an almost ample log cotangent sheaf on the universal family over a stratification of the moduli space of stable pairs. This is carried out in Appendix \ref{app:sheaves}.

\section{$ms$-integral points and subvarieties}\label{sec:subv}
In Section \ref{sec:deg_bound} we showed that $ms$-integral points lie in a subscheme whose degree is uniformly bounded (see Theorem \ref{th:deg_bound}). To conclude uniformity, one would hope to use an induction argument, by proving that $ms$-integral points that lie on a curve are stably integral for that curve, and then apply Theorem \ref{th:curves}. The purpose of this section is to show that the only obstruction to such an argument holding is the existence of contractible components of curves (see Proposition \ref{prop:no_contract}). We will then use this result in Section \ref{sec:descent} to show that the \emph{subvariety property} does hold under a positivity assumption on the log cotangent sheaf.

Let $(X,D)$ be a stable pair of dimension two and let $Y$ be an irreducible subvariety such that $Y \not\subset D$. We begin by studying the behaviour of $ms$-integral points lying on $Y$ when the pair $(Y,D_Y)$ is stable where $D_Y = (D \cap Y)$; in particular $D_Y$ is a reduced divisor.
 We want to show that an $ms$-integral point $P$ that lies in $Y$ is a stably integral point for $(Y,D_Y)$. This amounts to exhibiting a stable model $(\calY^s,\calD_Y^s)$ of $(Y,D)$ where $P$ is integral. Since $P$ is $ms$-integral in $(X,D)$ we are given a natural model of the two dimensional pair --  the good model $(\calX,\calD)$ where $P$ is integral. This defines a proper model of $(Y,D)$ as follows:
 
 \begin{definition}\label{def:inducedmodel}
   Let $(\calX,\calD)$ be a good model of a log canonical surface pair $(X,D)$ and let $Y \subset X$ be a proper irreducible curve. If $(Y,D_Y)$ is a stable pair, where $D_Y = (D \cap Y)$, then we call the closure of $(Y,D)$ in $(\calX,\calD)$ the \emph{induced model} of $Y$ inside $(\calX,\calD)$ and we denote it by $(\calY,\calD_Y)$.
 \end{definition}

 The induced model $(\calY,\calD)$ has stable generic fiber but in general might not be semi-stable itself! However the stable reduction theorem gives maps 
\[
  \begin{tikzcd}(\calY,\calD_Y) & (\calY^{ss},\calD_Y^{ss}) \ar[r]^{\phi} \ar[l] & (\calY^s,\calD_Y^s)\end{tikzcd}
\]
where $(\calY^{ss},\calD_Y^{ss})$ is the semistable reduction of $(\calY,\calD_Y)$, and all the maps are defined possibly in some finite extension of the base rings. Since the semistable reduction map is a composition of blow-ups, normalization, and base change under ramified covers we obtain the following.

\begin{lemma}
\label{lem:ss}
Let $(X,D)$ be a log canonical stable surface pair defined over $K$, and let $P$ be an $ms$-integral point with respect to a good model $(\calX,\calD)$ over $\Spec \calO_{K,S}$. Let $Y$ be an irreducible subvariety not contained in $D$ such that $(Y,D_Y)$ is a stable pointed curve. Then if $P \in Y$, (the image of) $P$ is integral in the semistable model $(\calY^{ss},\calD_Y^{ss})$ of the induced model $(\calY,\calD_Y)$.
\end{lemma}
\begin{proof}
  The proof follows by observing that the semistable reduction map sends $\calD_Y^{ss}$ to $\calD$.
\end{proof}

\begin{proposition}
\label{prop:no_contract}
Let $\phi: (\calY^{ss},\calD_Y^{ss}) \to (\calY^s,\calD_Y^s)$ be the stable reduction of the semi-stable model. If an $ms$-integral point $P$ of $(X,D)$ lies in $(Y,D_Y)$ and no irreducible component of a fiber of $(\calY^{ss},\calD_Y^{ss})$ containing the image of $P$ is contracted by $\phi$, then $P$ remains integral in $(\calY^s,\calD_Y^s)$.
\end{proposition}
\begin{proof}
 We will prove that the result holds under the weaker hypothesis that in each fiber of the induced model $(\calY,\calD_Y)$, the image of $P$ does not lie in a contractible component in a chain of rational curves containing a marked point. Consider the following diagram
  \[
    \begin{tikzcd}(\calY,\calD) \ar[d] & (\calY^{ss},\calD_Y^{ss}) \ar[l, "\sigma"] \ar[d] \ar[r, "\phi"] & (\calY^s,\calD_{Y}^s) \ar[d] \\ \Spec \calO_{K,S} \ar[bend left]{u}{P}& \Spec \calO_{K^{ss},S^{ss}} \ar[bend right]{u}[swap]{P} \ar[l] \ar[r] & \Spec \calO_{K^{ss},S^{ss}}  \ar[bend right]{u}[swap]{P}\end{tikzcd}
  \]
 By Lemma \ref{lem:ss}, $P$ is integral in the semistable model $\calY^{ss}$; we need to prove that the image of the contraction morphism $\phi(P)$ does not reduce to $\calD^s$ over any prime $\fp \notin S^{ss}$. Since this is local, we can fix a prime $\fp$ and work over the completion of $\calO_{K^{ss},S^{ss}}$ at $\fp$. The corresponding morphism, which we denote by $\phi_\fp$, is an isomorphism over the generic fiber $\calY^{ss}_\eta$ and a contraction on the special fiber $\calY^{ss}_\fp$. We have to prove that $\phi(P)$ does not specialize to a point lying in $\calD_Y^s$ in the fiber $\calY_\fp^s$.
 
 We argue by contradiction: assume that $\phi(P)$ intersects $\calD_Y^s$ in the fiber $\calY_\fp^s$. Since $\phi_\fp$ is the stable reduction of the fiber (see \cite[Chapter X]{acg}) we can assume that $\calY_\fp^{ss}$ was not already stable, otherwise $\phi_\fp$ would be an isomorphism, and we are assuming that $P$ is integral in the semistable model. Therefore, the fiber $\calY^{ss}_\fp$ is semistable but not stable and the map $\phi_\fp$ is a contraction of an exceptional chain $\Gamma$, i.e. a chain of rational curves which meets the rest of $\calY_\fp$ in at most two points (see the proof of \cite[Theorem 10.3.34 (b]{liu}).
 
 We are then reduced to the case in which the reduction of $P$ in  $\calY_\fp^s$ lies in $\Gamma$. By assumption, $P$ specializes to a point of $\calD_Y^s$, so there exists at least one component of $\Gamma$ containing at least one marked point, i.e. a point of $\calD_Y^{ss}$. We will show that this cannot happen.
 
 Recall that by hypothesis, the image of $P$ in the fiber $\calY_\fq$ of the model $\calY$, where $\fq$ is the prime lying under $\fp$, was not in any exceptional chain which meets the rest of $\calY_\fp$ in at most two points and contains a marked point. By Lemma \ref{lem:ss}, this implies that the same holds true in the semi-stable model. Therefore no component of $\Gamma$ can contain a marked point, exhibiting the contradiction. This proves that $\phi_\fp(\Gamma) = \phi_\fp(P)$ is disjoint from $\calD_Y^s$ and thus proves the proposition. 
\end{proof}

Lemma \ref{lem:ss} shows that any integral point $P: \Spec \calO_{K,S} \to (\calY,\calD_Y)$ remains integral in the semistable model. However, the image of $P$ under $\phi$ might intersect $\calD_Y^s$ in some fiber, so that $P$ is not integral in the stable model. To ensure that this will not happen, we will assume the log cotangent sheaf is almost ample.

By Corollary \ref{cor:slcample} the almost ampleness assumption ensures that all irreducible components of all fibers of $(\calY,\calD)$ are of log general type, as long as they are \emph{not contained in $D$ or the double locus of an slc fiber}. Since the stable reduction morphism contracts only components that are not of log general type, the almost ampleness assumption, combined with Proposition \ref{prop:no_contract}, implies that $\phi(P)$ remains integral in $(\calY^s)$. This shows that $ms$-integral points lying in stable subpairs \emph{outside the double locus} are stably integral in a stable pair of dimension two with almost ample log cotangent.

In particular, applying Proposition \ref{prop:no_contract} with a positivity assumption on the log cotangent of each fiber allows us to state the following Corollary.

\begin{corollary}\label{cor:descent1} Suppose that $(X,D)$ is a log canonical stable surface pair with good model $(\calX, \calD)$ such that each fiber has almost ample log cotangent. Let $(\calY, \calD_\calY)$ be the induced model of a stable curve $(Y, D_Y)$. If no fiber of $(\calY, \calD_\calY)$ lies in the double locus of a fiber of $(\calX, \calD)$, then $ms$-integral points of $(X,D)$ lying on $(Y, D_Y)$ are stably integral for $Y$, i.e. they satisfy the subvariety property. \end{corollary}
\begin{proof}
  Proposition \ref{prop:no_contract} ensures all fibers of $(\calY,\calD_\calY)$ are log general type, as we assumed no fiber lies in the double locus of a fiber of $(\calX,\calD)$. Given any $ms$-integral point $P$ of $(X,D)$ lying on $(Y,D_Y)$, this implies that the image of $P$ in the semistable model $(\calY^{ss},\calD_Y^{ss})$ either lies in an exceptional curve not containing points of $\calD_Y^{ss}$, or in a component of log general type. In particular either the component is not contracted by stable reduction, or it is contracted to a point not in $\calD_Y^s$.
\end{proof}

\section{$ms$-integral points in singular curves}\label{sec:sing}
Corollary \ref{cor:descent1} implies that $ms$-integral points lying on a stable pair $(Y,D_Y)$ satisfy the subvariety property provided that no component of a fiber of its induced model $(\calY, \calD_\calY)$ (see Definition \ref{def:inducedmodel}) lie in the double locus of a fiber of $(\calX, \calD)$. However, the $ms$-integral points might lie in non stable curves, (e.g. curves with worse than nodal singularities), in which the generic fiber of the induced model is not semi-stable and therefore we cannot apply stable reduction. Instead, we consider a stable map $Y' \to Y \hookrightarrow X$, and study how the $ms$-integral points on $Y'$ relate to those on $Y \subset X$.

\subsection{Singular curves}
Given a log canonical stable surface $(X,D)$, let $Y \subset X$ be an irreducible curve not contained in $D$ such that the pair $(Y, D_Y)$ is not stable. Therefore either $Y$ is singular, or $D_Y$ is non-reduced. If $Y$ is singular, but $D_Y$ is reduced, then we take the normalization $Y^\nu \to Y$ to obtain a stable map $Y^\nu \to X$, and proceed with Proposition \ref{prop:stable_map}. However, if $D_Y$ is nonreduced, then we give a construction of a stable map $f: Y' \to X$ where $Y'$ is a pre-stable curve, $f$ is a stable map, and the number of marked points in $Y'$ equals  $\deg(D_Y) = \deg(D \cap Y)$. In particular, we \emph{cannot} simply use the normalization as the degree equality will not be satisfied. This last condition will be necessary for studying curves that degenerate into the double locus of a degeneration of $X$ (see Section \ref{sec:descent} and Proposition \ref{prop:rational2}). We begin by giving an example that clarifies our motivation.

Recall that a \emph{stable map} (see e.g. \cite[Section 1.3]{AbGW}) is a map from a marked nodal connected projective curve to a projective variety such that the nodes are disjoint from the markings, and the group of automorphisms fixing the markings for a component contracted by the map is finite. 

\begin{example}
  Let $(X,D)$ be a log smooth stable surface with $D$ smooth and irreducible, and let $Y$ be a non-singular curve in $X$. Suppose that $\mathrm{Supp}(Y \cap D) = \{ P \}$ and that the multiplicity of intersection is 2. In this case $D_Y = D \cap Y = 2 P$. One natural way to construct a stable map $f: Y' \to X$ is to consider $Y' = Y$ where $f$ is the inclusion. However in this case the number of marked points on $Y'$ will be 1, while $\deg D_Y = \deg(D \cap Y) = 2$. On the other hand one can consider the surface pair $(X, D+Y)$ and the map $\pi: \widetilde{X} \to (X, D+Y)$, a compositions of two blow-ups centered at $P$ which gives a log-resolution of the pair. Then $\widetilde{Y} = \pi^*Y = Y_1 + E_1 + 2E_2$ is a non-reduced but nodal curve and $\widetilde{Y} \cap \pi_*^{-1}D$ is transverse; it is a point in $E_2$. We can take finite covers $t:\overline{X} \to \widetilde{X}$ branched along components of $\widetilde{Y}$ and take the (normalization of the) fibered product $\overline{Y} = \overline{X} \times_{\widetilde{X}} \widetilde{Y}$, which also gives a map $h: \overline{Y} \to X$. In this case, we take a single degree 2 cover $\overline{X} \to \widetilde{X}$ branched over $Y \cup E_1$. We then obtain a pre-stable curve $\overline{Y} \subset \overline{X}$ whose markings $D_{\overline{Y}}$ are given by the preimages of $D_{\widetilde{Y}}$, which are 2 points. Finally we can contract the non-stable components of the map $h$ -- namely the preimage of $E_1$ in $\overline{Y}$ -- and we get a curve $Y'$ which has two irreducible components, one isomorphic to $Y$ and the other a genus 1 component $E_2'$ containing the markings. Since we have contracted only unstable components for the map $h$ we get an induced map $f: Y' \to X$ which is stable and the number of markings of $Y'$ is precisely $\deg D_Y$.
\end{example}

\begin{construction}\label{sec:construction}Let $\pi: \widetilde{X} \to (X, D + Y)$ be a log resolution of the pair $(X, D+Y)$ and consider the curve $\widetilde{Y} = \pi^*Y$ (see Remark \ref{rmk:qfactorial}). By construction $\widetilde{Y}_{\text{red}}$ is a nodal curve. Consider the divisor $D_{\widetilde{Y}}$ on $\widetilde{Y}$ defined as $D_{\widetilde{Y}} = \widetilde{Y} \cap \pi_*^{-1}D$. Consider a composition of finite covers $t: \overline{X} \to \widetilde{X}$ branched along components of $\widetilde{Y}$ such that the fibered product $\overline{Y} = \overline{X} \times_{\widetilde{X}} \widetilde{Y}$ is a pre-stable curve. It comes with a map $h: \overline{Y} \to X$, and markings $D_{\overline{Y}}$ that are the preimages of $D_{\widetilde{Y}}$ under the cover.

Contracting any unstable components relative to the morphism $h: \overline{Y} \to X$ gives a curve $Y'$ with a map $\mu: \overline{Y} \to Y'$, and a stable map $f: Y' \to X$ whose image is $Y$, with markings given by the images $D_{Y'} = \mu(D_{\overline{Y}})$. By construction $\deg(D_{Y'}) = \deg(D_Y)$.\end{construction}

 \begin{remark}\label{rmk:qfactorial} Note that in the case in which $Y$ is not a $\Q$-Cartier divisor, we can consider a minimal $\Q$-factorialization $q: X' \to X$ (see e.g. \cite[Corollary 1.36]{singmmp}) and replace $Y$ by its strict transform in $X'$: this will not change the intersection $Y \cap D$ locally around $D$ since $(X,D)$ is log canonical. \end{remark}

 We first show that the stable map $f$ allows us to construct a nice semistable model of $Y'$.

\begin{proposition}
  \label{prop:stable_map}
The stable map $f$ induces a semistable model $\calY'$ of $Y'$ with a map $F: \calY' \to (\calX,\calD)$. Moreover, if each fiber of a good model of the log canonical stable surface $(X,D)$ has almost ample log cotangent then every positive dimensional component of every fiber of the stable reduction map $\calY' \to \calY''$ is mapped to $\calD$ or the double locus of any fiber of $(\calX, \calD)$ by $F$.
\end{proposition}
\begin{proof}
  Let $g$ and $n$  denote the genus and the number of markings of $Y'$ respectively, and let $\beta$ be the class in $H_2(\calX,\Z)$ of the closure of $Y$ in the good model of $(X,D)$. Then by \cite{AbOrt}, the stack of stable maps $\calM_{g,n}(\calX,\beta)$ is a proper Artin stack, and so the map $f$ admits a closure $F: \calY' \to \calX$ as a stable map over $\calO_{K,S}$ (after possibly extending $K$ and $S$). The map $F$ can be thought of a family of stable maps whose generic fiber is $f$.
  By definition, the curves appearing as fibers of $\calY'$ are connected, projective, and at worst nodal. Therefore this is a semistable model for $Y'$. 
  
 Let $\sigma: \calY' \to \calY''$ be the stable reduction map. We have to prove that any component contracted by $\sigma$ is sent to $\calD$ or to the double locus of a fiber of $(\calX,\calD)$ through $F$. Assume that there exists a prime $\fp$ not in $S$, and a irreducible component $Z$ of $\calY'_\fp = \calY' \times_{\Spec \calO_{K,S}} \F_\fp$ which makes $\calY'_\fp$ non-stable as a curve. This implies in particular, that $2g(Z) - 2 + n_\fp \leq 0$, where $n_\fp$ is the number of special points in $Z$. Since the map $F$ is stable, this implies that $F$ is not constant on $Z$, and therefore $F$ restricted to $Z$ gives a finite map $F: Z \to F(Z).$ In particular $F(Z)$ is irreducible. However, by Corollary \ref{cor:slcample}, every irreducible curve in any fiber of the good model $(\calX,\calD)$ not contained in $\calD$ or the double locus of any non-normal fiber is of log general type. Since $F(Z)$ is not of log general type, this implies that $F(Z)$ is contained in $\calD$ or the double locus of a fiber.
\end{proof}

The above proposition implies that the study of ms-integral points on singular curves on $X$ can be reduced to the study of integral points on the semistable model of the normalization endowed with the stable map to the good model. In analogy with Definition \ref{def:inducedmodel} we give the following:

\begin{definition}
  \label{def:induced_map}
Let $(\calX,\calD)$ be a good model of a log canonical surface pair $(X,D)$ and let $Y \subset X$ be a proper irreducible curve not contained in $D$. We call the model of the map $f: Y' \to X$ of Proposition \ref{prop:stable_map} the \emph{induced model} of $Y$ and denote it by $(\calY',\calD_{Y'})$, or $\calY'$. It comes endowed with a stable map $F: \calY' \to (\calX,\calD)$.
\end{definition}

\begin{lemma}
  \label{lem:ms-sing} 
Let $P: \Spec \calO_{K,S} \to (\calX,\calD)$ be an ms-integral point of $X$ lying in the smooth locus of a non-stable curve $Y \subset X$. Let $F: \calY' \to (\calX,\calD)$ be the induced model of $Y$ and assume that no fiber of $\calY'$ is mapped to the double locus of a fiber of $(\calX,\calD)$. Then $P$ lifts as an integral point on $\calY'$ and remains integral in the stable model $\calY''$ of $\calY'$. In particular, any ms-integral point lying in $Y$ is stably integral in the normalization $Y'$, i.e. $ms$-integral points satisfy the subvariety property.
\end{lemma}
\begin{proof}
  Since $P$ is a smooth point of $Y$, the point $P$ automatically lifts to a point of $Y'$. Since the model $\calY'$ is semistable, the point $P$ remains integral in the stable model if and only if it does not hit any component of $\calY'$ that gets contracted under stable reduction that contains a marked point.  By Proposition \ref{prop:stable_map}, any such component is mapped either to $D$ or to the double locus of a fiber of $(\calX,\calD)$ by $F$. Since  $P$ was $ms$-integral in $X$  the lift of $P$ on $\calY'$ cannot intersect any component mapped to $D$. On the other hand, by hypothesis there is no component of $\calY'$ mapped to the double locus of a fiber of $(\calX,\calD)$. This implies that the image of $P$ in the stable model of $Y'$ is integral. 
\end{proof}

\begin{remark}
  By definition of a stable map, in the semistable model of the normalization any lifting of an $ms$-integral cannot hit any vertical component that is mapped to $D$ outside $S$, since this will contradict the fact that the point is integral in the good model $(\calX,\calD)$.
\end{remark}

\section{Subvariety property for surfaces}\label{sec:descent}

By the work of Sections \ref{sec:subv} and \ref{sec:sing}, we see that we would be able to conclude uniformity (under a positivity assumption on the log cotangent) if no component of a fiber of the induced model $(\calY,\calD_\calY) \subset (\calX, \calD)$ which lies in the double locus of a fiber of $(\calX , \calD)$ is contractible, and we will show this in Proposition \ref{prop:rational} and \ref{prop:rational2}. We first recall a property of the double locus $\dbl$ of an slc  pair.

\begin{proposition}\cite[Sec. 5.2]{singmmp} \label{prop:kollar} The morphism $\pi: \dblv \to \dbl$ is generically finite of degree two, ramified at the pinch points. \end{proposition}

    \begin{prop}\label{prop:involution} Let $(X,D)$ be an slc surface pair. If $D \cap \dbl \neq \emptyset$, then the intersection must be a nodal point of $D$. \end{prop}
\begin{proof} 
There is a Galois involution $\tau: \overline{\Delta}^{\nu}_{dl} \to \overline{\Delta}_{dl}$, from the normalization of $\dblv$ to the normalization of $\dbl$ (see \cite[Section 5.1]{singmmp}). By \cite[Proposition 5.12]{singmmp} the different $(\overline{\Delta}^{\nu}_{dl}, \text{Diff}_{\overline{\Delta}^{\nu}_{dl}} D^\nu)$, where $\overline{\Delta}^{\nu}_{dl}$ is the normalization of $\dblv$, is $\tau$-invariant.  
In particular, $\text{Supp}(\text{Diff}_{\overline{\Delta}^{\nu}_{dl}} D^\nu)= \textrm{Supp}(D^\nu \cap \dblv)$ by \cite[Proposition 4.5]{singmmp}.  Since $(X,D)$ is slc, the intersection points $D^\nu \cap \dblv$ cannot contain the preimage of a pinch point. Therefore, for any point $p \in D \cap \dbl$,  the set $\{D^\nu \cap \dblv\}$ contains the whole fiber $\nu^{-1}(p)$, which consists of two points by Proposition \ref{prop:kollar}.  Therefore $p \in D$ is a node. 
\end{proof}

\begin{prop}\label{prop:rational}  Let $(\calX, \calD)$ be a good model of a log canonical stable surface pair such that every fiber has almost ample log cotangent. Let $(Y,E)$ be a stable curve inside the generic fiber of $(\calX, \calD)$, and let $(\calY, \calD_\calY)$ be its induced model (see Definition \ref{def:inducedmodel}). Then every positive dimensional component of every fiber of the stable reduction is not mapped to $\calD$. \end{prop}

\begin{proof} This is true for any component outside the double locus of every fiber of $(\calX, \calD)$ by Corollary \ref{cor:slcample}. Let $Z$ be a component of a fiber of $(\calY, \calD_\calY)$ which \emph{is} contained in the double locus of a fiber of $(\calX, \calD)$. If $Z \cap D = \emptyset$ the result is clear. Otherwise, by Proposition \ref{prop:involution} all of the points $Z \cap D$ are nodes of $D$. The strict transform of $Z$ in $(\calY^{ss}, \calD^{ss}_\calY)$ (obtained by blowing up these nodes) is thus disjoint from $D$, and therefore it is either not contracted, or it is contracted to a point not contained in $\calD$. \end{proof}

\begin{prop}\label{prop:rational2} Let $(\calX, \calD)$ be a good model of a log canonical stable surface pair such that every fiber has almost ample log cotangent. Let $(Y,E)$ be a unstable curve inside the generic fiber of $(\calX, \calD)$, and let $(\calY', \calD_{\calY'})$ be its induced model (see Definition \ref{def:induced_map}).  Then every positive dimensional component of every fiber of the stable reduction map $\calY' \to \calY''$ does not contain any marked points. \end{prop} 

\begin{proof} For components of $\calY'$ not mapped to the double locus of any fiber of $(\calX, \calD)$ the result was proven in Proposition \ref{prop:stable_map}. Suppose $Z$ is a component of a fiber $\calY'_b$ contracted by stable reduction which \emph{is} mapped to the double locus of a fiber of $(\calX, \calD)$. Since $F$ is a stable map, the restriction $F|_Z$ is a finite map. Since $Z$ gets contracted by stable reduction, we may assume that $g(Z) = 0$. By flatness we can further assume that $\calY'_b$ is not irreducible so that $Z$ has at least one special point (which is not a marked point). By Construction \ref{sec:construction}, the number of marked points on $Z$ is equal to $\deg\big(D \cap F(Z)\big)$. Since $F(Z)$ is contained in the double locus of $\calX_b$, by Proposition \ref{prop:involution} this number must be even.  If $\deg  = 0$ we are done, otherwise $\deg \geq 2$ but this will give at least two marked points, in addition to our special point which is not a marked point, finishing the proof. 
 \end{proof}

\begin{corollary}\label{cor:subvarietyprop} The $ms$-integral points on log canonical stable surface pairs with good model such that every fiber has almost ample log cotangent satisfy the subvariety property. \end{corollary}

\begin{proof} Combine the results of Corollary \ref{cor:descent1} and Lemma  \ref{lem:ms-sing} with Propositions \ref{prop:rational} and \ref{prop:rational2}. 
\end{proof}

\section{Uniformity for surfaces}\label{sec:secmainthm}
We begin by proving uniformity results for surfaces.
    \begin{theorem}
\label{prop:unif_2} Assume the Lang-Vojta Conjecture.
Let $(X,D) \to B$ be a stable family of surface pairs over $K$ with integral, openly canonical, and log canonical general fiber over a smooth projective variety $B$. If the log cotangent of each fiber is almost ample, then the cardinality of the set of $ms$-integral points of $(X_b,D_b)$ is uniformly bounded for all $b$ in $B(K)$. 
\end{theorem}
\begin{proof}
By Theorem \ref{th:deg_bound} there exists a closed subset $A_b$ containing all the $ms$-integral points, for every fiber $(X_b,D_b)$ and every rational point $b \in U \subset B$, where $U$ is an open subset of $B$. Moreover $d_b = \deg A_b$ is uniformly bounded by a constant $N$ (depending on the invariants defining the moduli space of stable pairs where the fibers of the family lie). We can write $A_b = A_{0,b} \cup A_{1,b} = A_0 \cup A_1$ where $A_0$ is the part of pure dimension zero and $A_1$ is the part of pure dimension one. Since $\deg A_b = d_b < N$ it follows that $A_0$ contains at most $N$ points.

 The dimension one part $A_1$, consists of irreducible curves $\{ \calC_i \}_{i\in I} \subset A_b$ for a finite set $I$ whose cardinality is uniformly bounded by $N$. We restrict our attention to the subset $J \subset I$ of curves $\calC_j$ which are not contained in $D_b$; note that the curves $\calC_i$ with $i \in I \setminus J$ do not contain any $ms$-integral points since they are contained in $D_b$. 
 
 For every $j \in J$, we let $E_j = (C_j \cap D_b)$. Then Corollary \ref{cor:subvarietyprop} implies that every $ms$-integral point of $(X_b, D_b)$ lying on $(C_j, E_j)$ is a stably-integral point. Moreover, since $N$ is uniformly bounded, there are only finitely many choices for the genus of the $\calC_j$,the degree of $E_j$, and the number of singular points of $C_j$. Applying Theorem \ref{th:curves} to the curves gives a bound on the maximum number of $ms$-integral points in the union of $\{ \calC_j \}$ for $j \in J$, which does not depend on $b$.

 Therefore the total number of $ms$-integral points in $A_b$ is uniformly bounded for every $b \in U$.

 To obtain the result for all fibers we apply Noetherian induction on the base: let $N_0$ be the bound obtained above for the cardinality of the set of $ms$-integral points on every fiber $X_b$ with $b$ lying in $U \subset B$. Let $B_1$ be the union of all irreducible components of $B \setm U$ whose generic point is openly log-canonical and consider the restricted family $(X_1,D_1) \to B_1$. Applying the procedure described in this proof to this new family gives a bound $N_1$ for the fiber lying in an open set $U_1 \subset B_1$. This process gives a sequence of $B_i$ which by Noetherian induction, since $\dim B_i > \dim B_{i+1}$, exhausts all points of $B$ corresponding to log-canonical stable pairs in a finite number of steps. Then $N = \max N_i$ is a uniform bound for the cardinality of the set of $ms$-integral points on every fiber.
\end{proof}

\begin{corollary}
\label{cor:unif_2}
Assume the Lang-Vojta Conjecture. Let $(X,D)$ be a log canonical stable surface pair with $D$ a $\Q$-Cartier divisor and good model $(\calX, \calD) \to B$. Suppose that each fiber of $(\calX, \calD)$ has almost ample log cotangent. 
Then there exists a constant $N = N(K,S,v)$ where $v$ is the volume of $(X,D)$, such that the set of $ms$-integral points of $(X,D)$ has cardinality at most $N$, i.e.
  \[
    \# (X \setm D)(\calO_{K,S}^{ms}) \leq N = N(K,S,v)
  \]
\end{corollary}

\begin{remark}\label{rmk:higherdim} Corollary \ref{cor:unif_2} can be extended to higher dimensions considering the set of $ms$-integral points lying on stable subvarieties. However in that case the conclusion is too weak, since on a variety of $\dim > 1$ there is a dense collection of non-stable subvarieties.\end{remark}

\begin{proof}
It is enough to apply Theorem \ref{prop:unif_2} to the tautological family of the moduli space $\calM_\Gamma$ of stable pairs of dimension 2, volume $v$ and coefficient set $I = \{ 1 \}$ (see \cite{kp}).
\end{proof}

\appendix
\section{Stack of stable pairs over $\Q$}\label{sec:stacks}

Suppose that we fix any moduli functor $\mathfrak{F}$ of stable log varieties of fixed dimension, volume and coefficient set as in \cite{kp};  Let
$
  (\calU, \calD) \to \calM_\Gamma
$
be the universal family, universal divisor and moduli stack representing the functor $\mathfrak{F}$. Before dealing with models we first need to prove that $\calM_\Gamma$ can be defined over $\Q$, and is still a Deligne-Mumford stack with projective coarse moduli space.

The main theorem of \cite{kp} states that the moduli stack $\calM_\Gamma$ exists over an algebraically closed field of characteristic 0. Here we show the existence and projectivity of the moduli space over $\Q$. Throughout this section, we denote the Deligne-Mumford moduli stack over $\overline{\Q}$, by $\calM_{\Gamma,\overline{\Q}}$. Define $\calM_{\Gamma, \Q}$ to be the moduli stack of stable pairs over $\Q$. The first result shows that this stack is Deligne-Mumford.

\begin{theorem}
\label{th:Q} 
The moduli stack $\calM_{\Gamma, \Q}$ of stable pairs over $\Q$ is Deligne-Mumford. 
\end{theorem}

\begin{proof}
Let $S = \Spec{\overline{\Q}}$ and let $(S_{\alpha}, u_{\alpha \beta})$ be a projective system of schemes so that $S = \varprojlim_{\alpha} S_{\alpha}$ and all $S_{\alpha} = \Spec K_{\alpha}$, where $K_{\alpha}$ is a finite extension of $\Q$. By \cite[Proposition 5.11]{kp}, the stack $\calM_{\Gamma, \overline{\Q}}$ of stable pairs is a Deligne-Mumford stack of finite type over $S$.  By \cite[Proposition 2.2]{Olsson}, there exists a Deligne-Mumford stack of finite type $\calM_{\Gamma, S_{\alpha}}$ over some $S_{\alpha}$ and an isomorphism $\calM_{\Gamma, \overline{\Q}} \cong \calM_{\Gamma, S_{\alpha}} \times_{S_{\alpha}} \Spec{\overline{\Q}}$ (if necessary, replace $S_{\alpha}$ so that $\calM_{\Gamma, S_{\alpha}}$ is a moduli stack). 

Since $\calM_{\Gamma, S_{\alpha}}$ is Deligne-Mumford, there is a surjective \'etale morphism $u: U \to \calM_{\Gamma, S_{\alpha}}$ where $U$ is a scheme.  We now wish to prove that $\calM_{\Gamma, \Q}$ is a Deligne-Mumford stack. To do so, consider the surjective morphism $\psi: U \to \calM_{\Gamma, S_{\alpha}} \to \calM_{\Gamma, \Q}$. By \cite[Tag05UL]{stacks-project}, $\calM_{\Gamma, \Q}$ is a Deligne-Mumford stack if the morphism $\psi$ is surjective, smooth, and representable by algebraic spaces. 

Then the morphism $\psi$ is smooth and surjective, since the morphism $u$ is smooth and surjective, and the morphism $\phi: \calM_{\Gamma, S_{\alpha}} \to \calM_{\Gamma, \Q}$ is \'{e}tale and surjective since the map $S_{\alpha} \to \Spec \Q$ is. To show representability of $\psi$, it suffices to check representability of $\phi$ as we know that the map $u$ is representable. Consider the following diagram, where $T$ is an algebraic space, and both squares are fiber product diagrams. 
\[ 
  \begin{tikzcd} \widetilde{T} \ar[r]^{} \ar[d]_{} & T \ar[d]^{} \\ 
  \calM_{\Gamma, S_{\alpha}} \ar[r]^{\phi} \ar[d] & \calM_{\Gamma, \Q} \ar[d] \\
   S_{\alpha} \ar[r]^{} & \Spec\Q \end{tikzcd}
\] To show that $\phi$ is representable, we must show that the fiber product $\widetilde{T}$ is an algebraic space whenever $T$ is. Since both squares are fiber products, $\widetilde{T} \cong S_{\alpha} \times_{\Spec \Q} T$ is also a fiber product, and is thus an algebraic space since it is the base change of an algebraic space by a field extension.
\end{proof}

For Appendix \ref{app}, where we give an alternative approach to models of stable pairs, we need to appeal to the projective coarse moduli space over $\Q$, so we show its projectivity. 

\begin{theorem}
\label{th:proj_coarse}
The coarse moduli space $M_{\Gamma, \Q}$ is a projective variety.
\end{theorem}

\begin{proof}
By \cite[Corollary 6.3]{kp}, the coarse moduli space $M_{\Gamma, \overline{\Q}}$ is a projective variety.  Let $\calL$ denote the ample line bundle on $\calM_{\Gamma, \overline{\Q}}$. While $\calL$ may not be defined over $\calM_{\Gamma, S_{\alpha}}$, by  \cite[Proposition 2.2]{Olsson}, we obtain some $S_{\beta}$ so that $\calL \to \calM_{\Gamma, \overline{\Q}}$ is defined over $S_{\beta}$, where $S_{\beta}$ again corresponds to a finite extension of $\Q$. Taking an algebraic extension containing the fields corresponding to both $S_{\alpha}$ and $S_{\beta}$, call the corresponding affine scheme $S_{\gamma}$, we obtain a stack $\calM_{\Gamma, S_{\gamma}}$ where the ample line bundle $\calL$ is defined.
Now we wish to show that $\calL$ is an ample line bundle exhibiting the coarse moduli space $M_{\Gamma, S_\gamma}$ as a projective variety. There Serre criterion for ampleness tells us that the higher cohomology vanishes on $M_{\Gamma, S_\gamma}$ for a high enough power of $\calL$. By \cite[Proposition 2.1]{conrad}, there is an isomorphism between the higher cohomology groups of $M_{\Gamma, \overline{\Q}}$ and $M_{\Gamma, S_\gamma}$, and thus shows that ampleness of the line bundle descends to $M_{\Gamma, S_\gamma}$.  As a result, the coarse moduli space $M_{\Gamma, S_{\gamma}}$ is a projective variety. Using Galois descent, we show that projectivity descends to $M_{\Gamma, \Q}$. 

Suppose $M_{\Gamma, S_{\gamma}}$ is defined over a field $K_{\gamma}$. Let $G$ be the finite Galois group corresponding to the finite extension $K_{\gamma} / \Q$. Then $\widetilde{\calL} =  (\displaystyle\otimes_{g \in G} \calL^g)$ gives a Galois invariant line bundle on $M_{\Gamma, S_{\gamma}}$, where $\calL^g$ denotes the pullback of the line bundle $\calL$ through the isomorphisms induced by $g \in G$. Since $\calL^g$ are ample line bundles for all $g \in G$, we see that the line bundle $\widetilde{\calL}$ is a Galois invariant ample line bundle. Moreover, by Galois descent, the Galois invariant line bundle $\widetilde{\calL}$ is pulled back from a line bundle $\calL'$ on $M_{\Gamma, \Q}$. Since the morphism $M_{\Gamma, S_{\gamma}} \to M_{\Gamma, \Q}$ is finite, the line bundle $\calL'$ is ample. This line bundle thus gives the desired projectivity of the coarse moduli space of $M_{\Gamma, \Q}$.
\end{proof}

\section{Sheaves on the universal family}\label{app:sheaves}
The goal of this section is to show that an almost ample log cotangent sheaf can be defined on the level of the universal family over the moduli stack of stable pairs.

\subsection{Simultaneous normalizations}
We begin by finding a sheaf on the normalization of fibers of the universal family, and to do so we use simultaneous normalizations of Chiang-Hsieh and Lipman \cite{ch} or Koll\'ar \cite{kolnorm}.  We recall the definition of a simultaneous normalization. 

\begin{definition}\label{def:normal} Let $f : X \to B$ be a morphism. A \emph{simultaneous normalization} of
$f$ is a morphism $\overline{\nu}: \overline{X}^\nu \to X$ such that: 
\begin{enumerate}
\item  $\overline{\nu}$ is finite and an isomorphism at the generic points of the fibers of $f$, and
\item $\overline{f} := f \circ \overline{\nu} : \overline{X}^\nu \to b$ is flat with geometrically normal fibers. \end{enumerate} \end{definition}

\begin{theorem}\cite[Theorem 12]{kolnorm}\label{thm:norm} Let $B$ be semi-normal, and let $f : X \to B$ be a projective morphism with generically reduced fibers of $\dim X_b = n$. The following are equivalent:
\begin{enumerate}
\item $X$ has a simultaneous normalization $\nu: \overline{X}^\nu \to X$
\item The Hilbert polynomial of the normalization of the fibers $\chi(X_B, \calO(tH))$ is
locally constant.
\end{enumerate}
\end{theorem}

\begin{remark}\label{rmk:normal} If $B$ is normal, then the total space $\overline{X}^\nu$ provided by the simultaneous normalization (if it exists) coincides with the normalization of the total space (\cite[Theorem 2.3]{ch}). \end{remark}

We need a simultaneous normalization as above, but in the pairs setting. The only complication is to determine what divisor to choose on $\overline{X}^\nu$ which corresponds to the double locus on fibers. 

\begin{definition}\label{def:conductor} 
Let $X$ be a reduced scheme and let $\nu: X^\nu \to X$ be its normalization. The \emph{conductor ideal sheaf} is the annihilator sheaf $\mathrm{Ann}_{\calO_X}(\nu_*\calO_{X^\nu} / \calO_X),$
is the largest ideal sheaf on $X$ that is also an ideal sheaf on $X^\nu$ .\end{definition}

\begin{remark}\label{rmk:conductor} If $X$ is slc, then the conductor ideal sheaf of $X^\nu$ corresponds to a divisor (see \cite[Remark 4.5]{kss}). We will denote the divisor corresponding to the conductor by $\calC$. \end{remark}

\begin{remark}\label{rmk:extensionnorm} We now discuss the extension of Theorem \ref{thm:norm} for pairs. Given a stable family $(X, D) \to B$ with log canonical general fiber over a semi-normal base $B$, if there exists a simultaneous normalization $\overline{\nu}: \overline{X}^\nu \to X$, then the total space is given by $({\overline{X}^\nu}, \overline{D}^\nu + \overline{\Delta}^{\nu}_{dl})$. Moreover, if the base $S$ is normal, then this total space is the normalization of $(X,D)$. In particular, the divisor $\overline{\Delta}^{\nu}_{dl}$, which cuts out the locus corresponding to the double locus of the normalization, corresponds to the double locus of the the normalization of each fiber. \end{remark}

\subsection{Universal family construction}

\begin{remark}\label{strata} We sketch the approach to stratifying $(\calU, \calD) \to \calM_{\Gamma}$. Let $(\fU, \fD) \to \fM$ be models of the universal family and modulil space. We stratify so that:
\begin{enumerate}
\item Each strata ${\fMb}$ satisfies Theorem \ref{thm:norm} (i.e. locally constant Hilbert polynomial), giving a simultaneous normalization $(\fU_b, \fD_b) \to \fMb$. 
\item Each strata is normal, to apply Remark \ref{rmk:normal}, and then take a flattening stratification to ensure that the conductor divisor (Remark \ref{rmk:conductor}) is flat over the base.
\end{enumerate}  \end{remark}

First we prove existence of a quasicoherent sheaf on each strata $(\calU_b, \calD_b) \to \calM_b$, and then show it also descends to a quasicoherent sheaf on any model $(\fU_b, \fD_b) \to \fMb$. Fix a strata $(\calU_b, \calD_b) \to \calM_b$ of the moduli stack of stable pairs with fixed invariants $(\calU, \calD) \to \calM_{\Gamma}$.

\begin{lemma}
 \label{lem:cot}There exists a coherent sheaf $\calF_b$ on $(\calU_b, \calD_b)$, up to finite base change, such that the restriction of $\calF_b$ to any fiber is isomorphic to the log cotangent sheaf. \end{lemma} 

\begin{proof}
By Remark \ref{strata} we can assume $\calM_b$ is normal. Since $\calU_B$ and  $\calM_B$ are Deligne-Mumford stacks, there exist surjective \'etale maps from schemes making the follow commute:

  \[
    \begin{tikzcd} V_{\calU_b} \ar[r] \ar[d]  & \calU_b  \ar[d] \\
    	       V_{\calM_b} \ar[r] & \calM_b  \end{tikzcd}
\] 

Since $\calM_b$ is normal, by Remark \ref{rmk:normal} the simultaneous normalization $\nu$ of $V_{\calU_b} \to V_{\calM_b}$ is the normalization. Denote the new total space by $\overline{V}_{\calU_b}$. By \cite[Lemma A.4]{dori}, $\overline{V}_{\calU_b}$ provides an atlas for a normal Deligne-Mumford stack $\overline{\calU}_b$. Furthermore, obtain a universal divisor $\overline{\calD}_b$, which is $\calD^\nu_b + \calI^\nu_b$, where $\calI^{\nu}_b$ denotes the divisor corresponding to the conductor ideal sheaf.

Since the map $\overline{V}_{\calU_b} \to \overline{\calU}_b$ is faithfully flat,  by \cite[7.18]{vistoli}, to define a sheaf $\calF_b$ on $\overline{\calU}_b$ is equivalent to giving a sheaf over $\overline{V}_{\calU_n}$ with \emph{descent data}. The scheme $\overline{V}_{\calU_b}$ comes equipped with a coherent sheaf, namely the relative log cotangent with respect to the map $\overline{V}_{\calU_b} \to V_{\calM_b}$.  Denoting this sheaf by $\calG_b$, we want to exhibit descent data yielding the the desired sheaf on $(\overline{\calU}_b, \overline{\calD}_b)$. We have a presentation \[
  \begin{tikzcd} V_R \ar[r, shift left, "p_1"] \ar[r, shift right, swap, "p_2"] & \overline{V}_{\calU_b} \ar[r] & \overline{\calU}_b \end{tikzcd}.
\]
This gives two sheaves on $V_R$, namely $p_1^*\calG_b$ and $p_2^*\calG_b$. The sheaf $\calG_b$ satisfies the universal property of the log cotangent sheaf (i.e. it is universal among derivations with simple poles on the divisor), implying the existence of an isomorphism 
$ \tau: p_1^*\calG_b\to p_2^*\calG_b$
In particular $\tau$ is a gluing datum for $\calG_b$. Note that we have the following diagram.
\[
  \begin{tikzcd} V:= V_R \times_{\overline{V}_{\calU_b}} V_R  \ar[r, shift left = .35em] \ar[r] \ar[r, shift right = .35em] & V_R \ar[r, shift left, "p_1"] \ar[r, shift right, swap, "p_2"] & \overline{V}_{\calU_b} \ar[r] & \overline{\calU}_b \end{tikzcd}
\] \
 To show that $\tau$ is a descent datum, we need to show that the fiber product of the three pullbacks of $\tau$ through the naturally defined maps $V \to V_R$ satisfy a cocycle relation. This follows from  the universal property of the log cotangent (\cite[7.20: (ii]{vistoli}).
Thus there is a sheaf $\calF_b$ on $(\overline{\calU}_b, \overline{\calD}_b)$ that coincides fiberwise with the log cotangent sheaf of the corresponding stable pair. Taking the reflexive hull of $\calF_b$ gives the desired sheaf on $(\overline{\calU}_b, \overline{\calD}_b)$.
\end{proof}

The existence of the sheaf $\calF_b$ on the universal family over any fixed strata implies that there is a choice of a finite extension $S' \supset S$ such that the log cotangent sheaf of every fiber of the fixed model of this universal family is almost ample.

\begin{lemma}
  \label{lem:ample_stack1}
Suppose that every stable pair of dimension two with fixed invariants $\Gamma$ defined over $K$ has almost ample log cotangent. Then there is a finite set of places $S$ and a sheaf $\calF_b$, possibly up to finite base change, of $(\mathfrak{U}_b, \fD_b) \to \mathfrak{M}_{b,\Gamma}$  with fixed invariants $\Gamma$ over $\calO_{K,S}$ which is relatively almost ample away from $S$. \end{lemma}
\begin{proof}
As above, let $\calM_b$ be a (normal) strata of the moduli stack of stable pairs, let $(\calU_b, \calD_b)$ be the universal family. By Lemma \ref{lem:cot}, the family $(\calU_b, \calD_b) \to \calM_b$ comes with a coherent sheaf, the log cotangent sheaf $\calF_b$, which by assumption is almost ample. We need to show that this sheaf extends to an almost ample coherent sheaf on any model.   Consider the following presentation, \[
    \begin{tikzcd} R_{\calU_b} \ar[d, shift left] \ar[d, shift right] & R_{{\fU}_b} \ar[d, shift left] \ar[d, shift right] \\ V_{\calU_b} \ar[d] & V_{{\fU}_b} \ar[d] \\ {\calU}_b \ar[d] \ar[r]& {\fU}_b \ar[d] \\ \calM_b  \ar[r] & \fMb \end{tikzcd}.
\] By construction of $\calF_b$ in Lemma \ref{lem:cot}, we have an almost ample coherent sheaf, which we call $F_b$, on the atlas $V_{{\calU}_b}$ and therefore, by pushing forward, we have a coherent sheaf on $V_{{\fU}_b}$. The two pullbacks of this sheaf to the presentation $R_{{\fU}_B}$ are compatible since they agree on the generic fiber and therefore give a descent datum for the sheaf. This implies that the sheaf $\calF_b$ extends over an open subset of the base $\fMb$. 

Recall that the sheaf $\calF_b$ is almost ample on every fiber of $(\fU_b, \fD_b)$ (possibly up to finite base change). Hence, after a finite extension $S' \supset S$ if necessary (but noting that this extension \emph{does not} depend on $(X,D)$), we can assume that outside $S'$, the log cotangent sheaf extends to a relatively almost ample sheaf $\calF_b$ over any model of the universal family.
\end{proof}

\begin{corollary}\label{cor:ample_stack}
If every stable surface pair of fixed invariants $\Gamma$ defined over $K$ has almost ample log cotangent, then there is a finite set $S' \supset S$ such that  $\calF_b$ is relatively almost ample away from $S$. 
\end{corollary}
\begin{proof} Take the (finite) union of the $S$ appearing in Lemma \ref{lem:ample_stack1} for each stratum, noting that we can make the base changes of the above Lemma ``global'', since the moduli space is of finite type.
\end{proof}

\section{Coarse moduli-stably integral points}
\label{app}
In this appendix we show how to define $ms$-integral points without appealing to good models (see Section \ref{sec:models}). The definition we state here is weaker, in the sense that it relies on models of the coarse moduli space of stable pairs instead of models of the stack. However, it has the advantage that such models can be proven to exist unconditionally. We start by recalling the following definition:

\begin{definition}\cite{Deligne}
  \label{def:coarse}
  A \emph{coarse moduli space} for a stack $\fM$ over a base scheme $S$ is an algebraic space $[\fM]$ over $S$ with a $S$-morphism $\phi_{\fM}: \fM \to [\fM]$ such that:
  \begin{enumerate}
    \item Every morphism $\fM \to X$, for an algebraic space $X$, factors uniquely through $\phi_{\fM}$;
    \item For every geometric point $\bar{s} \in S(\overline{k})$, $\pi$ induces a bijection between isomorphism classes of $\fM$ over $\bar{s}$ and $[\fM(\bar{s}]$.
  \end{enumerate}
\end{definition}

Given a stable family $(X, D)\to B$ with fixed volume, dimension and coefficient set (as in \cite{kp}), any moduli functor $\mathfrak{F}$ for which $(X,D)$ is an object of $\mathfrak{F}(B)$ is proper, and any algebraic space which is a coarse moduli space for the functor is a projective variety (\cite[Theorem 1.1]{kp}). In particular, if we fix one of such functor (for example the one described in \cite[Definition 5.6]{kp}, which in addition is a Deligne-Mumford-stack of finite type over any algebraic closed field of char 0), then the corresponding moduli stack of stable pairs $\calM_{\Gamma}$, for fixed geometric invariants, is projective and possesses a universal family $(\calU,\calD)$ with the property that the following diagram commutes:
\[
  \begin{tikzcd} (X,D) \ar[r]^{} \ar[d]_{} & (\calU,\calD) \ar[d]^{} \\ 
    B \ar[r]^{} & \calM_\Gamma\end{tikzcd} 
\] 
By Theorems \ref{th:Q} and \ref{th:proj_coarse} the above diagram holds over $\Q$, and therefore over any number field, and the corresponding coarse moduli spaces are projective over the same ground field. It is not known in general whether the moduli stacks $(\calU,\calD) \to \calM_{\Gamma}$ possess models over Dedekind domains which are moduli stacks for stable pairs over such domains (although it is possible to find specific models of the such stacks using limit methods for algebraic stacks, see \cite{Olsson} or \cite{Rydh}). This would make the previous diagram hold when the family $(X,D) \to B$ has a model over the ring of integers of a number field. However, allowing finite extensions of $S$, one can identify specific models for the coarse moduli spaces that depend only on the stacks, the number field, and the set of places.
      
\begin{theorem}
  Given a number field $K$ and a finite set of places $S$ there exists $S^m \supset S$ such that the coarse moduli spaces $[\calM_\Gamma], [\calU]$, and $[\calD]$ admit proper models over $\Spec \calO_{K,S^m}$. Moreover, the set $S^m$ depends only on the (models of the) underlying projective varieties.
  \label{th:models}
\end{theorem}
\begin{proof}
  By definition the stack $\calM_{\Gamma}$ has an underlying coarse moduli space which is a projective variety, and therefore it is the locus of the common zeroes of finitely many polynomials $g_1,\dots,g_n$ with coefficients in $K$. It then follows that since the number of overall coefficients are finite, there exists a (minimal) set of places $T$ such that the $g_i$ are polynomials with $T$-integral coefficients. For such $T$, the coarse moduli space $[\calM_\Gamma]$ has a model: $[\calM_\Gamma]^m$ over $\calO_{K,T}$. Since the model is projective by construction, it is also proper. Applying the same procedure to $\calU$ and $\calD$, one finds a (possibly larger) set of places $T$ such that $([\calU],[\calD])\to [M_\Gamma]$ have a model $([\calU]^m,[\calD]^m)\to [\calM_\Gamma]^m$ over $\calO_{K,T}$. Define $S^m = S \cup T$; by definition all the coarse moduli spaces have models over $\calO_{K,S^m}$ and, at most after enlarging the set $S^m$, the moduli map $(\calU,\calD) \to \calM_{\Gamma}$ extends to a map on the models.
\end{proof}

Given the model defined in Theorem \ref{th:models}, we can define coarse moduli stably $S$-integral points with respect to the choice of the model we made.

\begin{definition}
  \label{def:m-stab}
  Let $P$ be a $K$-rational point of a stable pair $(X,D)$ defined over $K$, and let $\overline{P}$ denote the corresponding maps to the moduli space
  \[ 
    \begin{tikzcd} \overline{P}: \Spec K \ar[r] & (X,D) \ar[r]^{} \ar[d]_{} & (\calU,\calD) \ar[d]^{} \\ 
    \ & \Spec K \ar[r]^{} & \calM_{\Gamma}\end{tikzcd} \] 
    We say that $P$ is \emph{coarse moduli-stably $S$-integral over $K$}, or $cms$-integral, if the image is $(S^m,\calD)$ integral in $\calU$. We denote by $X(\calO_{K,S}^{\text{cms}})$ the set of all $cms$ $S$-integral points over $K$.
\end{definition}

Note that the previous definition depends on the choice of models made in Theorem \ref{th:models}. This would be canonical if the existence of a moduli stack $\calM_\Gamma$ over $\Spec \calO_{K,S}$ was known.

The definition of $cms$-integral points depend on models of the coarse moduli space. As we will use results comparing integral points of fibered powers of families of stable pairs with integral points of the pair itself, we need to ensure that Definition \ref{def:m-stab} is compatible with fibered powers. In general fibered powers do not commute with the formation of coarse moduli spaces. However a weaker property of base change holds: for any map of schemes $S' \to S$, the following map:$
  \iota: [\calM_\Gamma \otimes_S S'] \to [\calM_\Gamma] \otimes_S S'$
is universally injective. Recall that a universally injective morphism is a morphism that is injective on $K$-valued points for every field $K$.

We want to compare the coarse moduli space of the fiber product $[\calM_\Gamma \times \calM_\Gamma]$ with the fiber product of the coarse moduli space $[\calM_\Gamma] \times [\calM_\Gamma]$ (every product is over the fixed base, $\Spec K$) in the case in which the coarse moduli space $[\fM]$ is a scheme. There is a map between the fibered product of the stack to the fibered power of the coarse moduli space given by composition as follows:
\[
 \alpha: \calM_\Gamma \times \calM_\Gamma \to \calM_\Gamma \times [\calM_\Gamma] \to [\calM_\Gamma \times [\calM_\Gamma]] \to [\calM_\Gamma] \times [\calM_\Gamma].
\]
By the first property of coarse moduli spaces, this defines a map:
$
  \beta_2: [\calM_\Gamma \times \calM_\Gamma] \to [\calM_\Gamma] \times [\calM_\Gamma],
$
which, inductively, defines a map $ \beta_n: [\calM_\Gamma^n] \to [\calM_\Gamma]^n.$ 

However this map is not injective in general, since the group of automorphisms of the fiber product is, in general, strictly contained in the product of the automorphism group. In particular, if (\'{e}tale) locally over the coarse moduli space, $\calM_\Gamma$ is presented locally as $X/G$, $\calM_\Gamma \times \calM_\Gamma$ is presented by the quotient of $X \times X$ by a finite group $G'$, which in general will be a subgroup of the $\Aut(X \times X)$. On the other hand, the coarse moduli space $[\calM_\Gamma] \times [\calM_\Gamma]$ will look locally like $X/G \times X/G$. Since $G' \subset G \times G$ this gives locally a finite map $\beta: [\calM_\Gamma \times \calM_\Gamma] \to [\calM_\Gamma] \times [\calM_\Gamma]$.

\begin{definition}
  \label{def:fp}
  Let $(X,D)$ be a stable pair defined over a $K$, let $\calO_{K,S}$ be the ring of $S$-integers, and let $n$ be a positive integer. A rational point $P \in (X^n,D_n)$ is \emph{coarse moduli stably $(S,D_n)$-integral} if $\beta(P)$ is $(S_m,\calD)$ integral in the model of the coarse moduli space given by the fiber product of the models defined in Theorem \ref{th:models}. \end{definition}
  
   Explicitly, this happens if the pair over $\Spec K$ comes with a commutative diagram
    \[ 
      \begin{tikzcd} (X,D) \ar[r]^{} \ar[d]_{} & (\calU, \calD) \ar[d]^{} \\ 
       \Spec K \ar[r]^{} & \calM_\Gamma \end{tikzcd} 
    \] 
    for a specific moduli stack $\calM_{\Gamma}$. Let $[\calU]^m$ be the model of the coarse moduli space $[\calU]$ over $\calO_{S^m,K}$: then $([\calU]^m)^n$ is a model of $[\calU]^n$. One can prove Theorem \ref{th:deg_bound} using $cms$-integral points instead of $ms$-integral points. However, to prove uniformity, one would need to extend the subvariety property	 to $cms$-integral points, which a priori is harder to achieve since there is no good model to refer to.

\section{Hyperbolicity}\label{sec:apphyp}
In this section we give an example that shows that hyperbolicity, in the sense of having all subvarieties of log general type, is not a closed condition.

\begin{proposition}\label{smoothing}
	Let $C \subset \PP^3$ be a curve and suppose that there exists a smoothing of $C$ in $\PP^3$.  If  $X \subset \PP^3$ is a surface of sufficiently high degree such that $C \subset X$, then there exists a smoothing of $X$ in $\PP^3$ that contains a smoothing of $C$.
\end{proposition}
\begin{proof}
	Let $T = \Spec R$ with $R$ a DVR and $\calC_T$ be a smoothing of $C$ in $\PP^3$ such that $C \cong \calC_0$.  Consider the sequence of sheaves over $T$
	\[ 0 \to \calI_{C_T} \to \calO_{\PP^3_T} \to \calO_{C_T} \to 0 .\]
	Twisting by $\calO(d)$, for any $d$, we have 
	\[ 0 \to \calI_{C_T}(d) \to \calO_{\PP^3_T}(d) \to \calO_{C_T}(d) \to 0 .\]
	
	Pushing forward along $\pi: \PP^3_T \to T$, we have
	\[ 0 \to \pi_* \calI_{C_T}(d) \to \pi_* \calO_{\PP^3_T}(d) \to \pi_* \calO_{C_T}(d) \to R^1\pi_* \calI_{C_T}(d), \]
	but for all $d \gg 0$ we have $H^1(\PP^3_t, \calI_{C_t}(d)) = 0$, hence by Cohomology and Base Change \cite[Theorem 12.11]{hartshorne}, $R^1\pi_* \calI_{C_T}(d)= 0$.  Therefore we get an exact sequence \[0 \to \pi_* \calI_{C_T}(d) \to \pi_* \calO_{\PP^3_T}(d) \to \pi_* \calO_{C_T}(d) \to 0,\]
	which shows that $\pi_* \calI_{C_T}(d)$ is a vector bundle over $T$.  Let $X$ be a surface with $C \subset X$, considered as an element of $H^0(\PP^3_0, \calI_{C_0}(d))$.  Let $\sigma$ be a general section of $\pi_* \calI_{C_T}(d)$ such that $\sigma_0 \cong X$.  To show that a general $\sigma$ gives a smoothing of $X$, it suffices to show that the general surface $X_t$ containing $\calC_t$ is smooth.  However, $\calC_t$ is a smooth curve, so there exists a surface $X_t$ containing $\calC_t$ and the general such $S_t$ is smooth (see e.g. \cite[Theorem 1.1, Remarks 1.2(a]{smoothing}). 
\end{proof}

The following Example \ref{ex:counterexample} shows the existence of a stable family $(X,D) \to B$ over a curve, whose generic fiber $(X_\eta, D_\eta)$ is smooth, and hyperbolic, in the sense that it does not contain any non-stable curve, but with special fiber $(X_0, D_0)$ such that:
\begin{itemize} 
\item all curves not contained in the double locus are of log general type, 
\item there exists a unique rational curve $C_{rat}$ in the double locus which only meets $D_0$ at a single (nodal) point, and 
\item the curve $C_{rat}$ is a component of the limit of a family of smooth stable curves in $(X,D) \to B$. 
\end{itemize}
In particular, the special fiber is not hyperbolic since it contains a curve which is \emph{not} of log general type. 

We recall that a surface $X$ is said to be \emph{semismooth} if the closed points of $X$ are smooth, double normal crossings, or pinch points.

\begin{example}\label{ex:counterexample}	Let $C = C_0 \cup C_g \subset \PP^3$ be the union of a line $C_0$ and a sufficiently general curve $C_g$ of genus $g$ such that $C_0 \cap C_g$ is a single point $p$.  
	
	We can find a semismooth surface $X$ containing $C$ such that $C_0$ is contained in the double locus $\dbl \subset X$.  First, take any smooth surface $X_1$ of sufficiently high degree containing $C_g$ (e.g. see \cite{smoothing}) and any surface $X_2$ containing $C_0$.  If $X_3$ is a sufficiently general surface containing $C_0$, the surface $X = X_1 \cup X_2 \cup X_3$ is semismooth, contains $C$, and has $C_0$ contained in the double locus.  
	
	Note that we can find a smoothing of $C$ in $\PP^3$; this is guaranteed by Brill-Noether theory (see e.g. \cite{harris}). Therefore, by Proposition \ref{smoothing} we can simultaneously smooth $X$ and $C$, i.e. we can find a smoothing $\calX_T$ of $\calX_0 \cong X$ in $\PP^3$ that contains a smoothing $\calC_T$ of $C = \calC_0$.  If $D \subset \PP^3$ is a general hyperplane section, the family $(\calX_T, D \times T)$ is a family of semismooth pairs such that each fiber $(\calX_{t}, D)$ is stable.
	
	By \cite[Corollary 1.3(c)]{smoothing}, the Picard group of the very general surface $X \subset \PP^3$ of degree $d\gg 0$ containing $C$ is generated by $C$ and $H$. By an adjunction calculation, the surface $X$ cannot contain any lines or conics. In particular, all subvarieties of $(X,D)$ are of log general type. 
	
	However, the family of curves $\big(\calC_T, (\calC_T \cap D \times T)\big)$ is not stable.  In the central fiber $\calC_0$, the rational component $C_0$ meets $C_g$ in a single point $p$ and $C_0 \cap D$ is a double point $q$.  Therefore, in stable reduction we blowup up $q$ and the strict transform of $C_0$ is contracted.  
\end{example}

%%%%%%%%%%%%%%%%%%%%%%%%%%%%%%%%%%%%%%%%%%%%%%%%%%%%%%%%%%%%%%%
%	BIBLIOGRAPHY
%%%%%%%%%%%%%%%%%%%%%%%%%%%%%%%%%%%%%%%%%%%%%%%%%%%%%%%%%%%%%%%
\bibliographystyle{alpha}	% (uses file "plain.bst")
\bibliography{uniformity}
%%%%%%%%%%%%%%%%%%%%%%%%%%%%%%%%%%%%%%%%%%%%%%%%%%%%%%%%%%%%%%%
%%%%%%%%%%%%%%%%%%%%%%%%%%%%%%%%%%%%%%%%%%%%%%%%%%%%%%%%%%%%%%%

\end{document}